\documentclass[reqno, eucal,16pt]{amsart}

\usepackage[mathscr]{eucal} %  use \EuScript (\mathcal unchanged)
\usepackage[dvipdfmx]{graphicx}
\usepackage[dvips]{color}
\usepackage{amsmath,amsfonts,amssymb,amsthm,amscd}
\usepackage{epic}
\usepackage{eepic}
\usepackage{longtable}
\usepackage{array}
\usepackage{fourier}

\usepackage[font=small]{caption}
\usepackage{ytableau}
\usepackage[vcentermath, enableskew]{youngtab}
\usepackage{hyperref}
\hypersetup{colorlinks,linkcolor={red},citecolor={olive},urlcolor={red}}

\usepackage{enumitem}

\newtheorem{theorem}{Theorem}[section]
\newtheorem{lemma}[theorem]{Lemma}

\newtheorem{proposition}[theorem]{Proposition}
\newtheorem{corollary}[theorem]{Corollary}

\theoremstyle{definition}

\newtheorem{example}[theorem]{Example}
\newtheorem{remark}[theorem]{Remark}

\newcommand{\Z}{{\mathbb Z}}

\def\eps{\varepsilon}
\def\PP{\mathbb{P}}
\def\EE{\mathbb{E}}
\def\wt{\textup{wt}}

\begin{document}

\title[Large deviations and one-sided scaling limit of multicolor BBS]{Large deviations and one-sided scaling limit \\ of randomized multicolor box-ball system}

\author{Atsuo Kuniba}
\address{Atsuo Kuniba, Institute of Physics, 
University of Tokyo, Komaba, Tokyo 153-8902, Japan}
\email{\texttt{atsuo.s.kuniba@gmail.com}}

\author{Hanbaek Lyu}
\address{Hanbaek Lyu, Department of Mathematics, 
University of California, Los Angeles, CA 90095, USA}
\email{\texttt{colourgraph@gmail.com}}

\maketitle

\vspace{0.5cm}
\begin{center}{\bf Abstract}
\end{center}

The basic $\kappa$-color box-ball (BBS) system is an integrable 
cellular automaton on one dimensional lattice 
whose local states take $\{0,1,\cdots,\kappa \}$ with $0$ regarded as an empty box. 
The time evolution is defined by a combinatorial rule of quantum group theoretical origin, 
and the complete set of conserved quantities
is given by a $\kappa$-tuple of Young diagrams. 
In the randomized BBS, a probability distribution on 
$\{0,1,\cdots,\kappa \}$ to independently fill the 
consecutive $n$ sites in the initial state induces a highly 
nontrivial probability measure on the $\kappa$-tuple of 
those invariant Young diagrams.
In a recent work \cite{kuniba2018randomized}, 
their large $n$ `equilibrium shape' 
has been determined in terms of Schur polynomials  
by a Markov chain method and also by a very different approach of 
Thermodynamic Bethe Ansatz (TBA).
In this paper, we establish a large deviations principle 
for the row lengths of the invariant Young diagrams. 
As a corollary, 
they are shown to converge almost surely to the equilibrium shape 
at an exponential rate. We also refine the TBA analysis and
obtain the exact scaling form of the vacancy, the row length and the 
column multiplicity, which exhibit nontrivial factorization 
in a one-parameter specialization. 
\vspace{0.4cm}

\section{Introduction}
\label{Introduction}

\subsection{The basic $\kappa$-color BBS}
\label{subsection:intro1}

The \textit{basic $\kappa$-color box-ball system} (BBS) is a cellular automaton on the half-integer lattice $\mathbb{N}$. At each discrete time $t\ge 0$, the system configuration is given by a coloring $X_{t}:\mathbb{N}\rightarrow \{0,1,\cdots,\kappa \}$ with finite support. When $X_{t}(x)=i$, we say that site $x$ is \textit{unoccupied} at time $t$ if $i=0$, and \textit{occupied with a ball of color $i$} at time $t$ if $1\le i \le \kappa$. To define the time evolution rule, for each $1\le a \le \kappa$, let $K_{a}$ be the operator on the set $\{0,1,\cdots,\kappa \}^{\mathbb{N}}$ of all $(\kappa+1)$-colorings on $\mathbb{N}$ defined as follows:
\begin{description}
	\item{(i)} Label the balls of color $a$ from left to right as $a_{1},a_{2},\cdots,a_{m}$. 
	\item{(ii)} Starting from $k=1$ to $m$, move ball $a_{k}$ to the leftmost unoccupied site to its right. 
\end{description}  
Now the time evolution $(X_{t})_{t\ge 0}$ of the basic $\kappa$-color BBS is given by 
\begin{equation}\label{eq:BBS_def_nonlocal}
X_{t+1} =  K_{1} \circ K_{2}\circ  \cdots \circ K_{\kappa}(X_{t}) \quad \forall t\ge 0.
\end{equation}
A typical 4-color BBS trajectory is shown below.
\begin{eqnarray*}
	t=0: && 1 1 2 1 4 0 1 0 1 2 1 4 2 0 4 4 2 0 1 2 0 0 0 0 0 0 0 0 0 0 0 0 0 0 0	0 0 0 0 0 0 0 0 0 0 0 0 0\\
	t=1: && 0 0 1 0 2 4 0 1 0 1 0 2 1 4 2 1 1 4 0 1 4 2 2 0 0 0 0 0 0 0 0 0 0 0 0 0 0 0	0 0 0 0 0 0 0 0 0 0   \\
	t=2: && 0 0 0 1 0 2 4 0 1 0 1 0 0 2 1 0 0 2 4 0 1 1 1 4 4 2 2 0 0 0 0 0 0 0 0 0 0 0 0 0 0 0	0 0 0 0 0 0\\
	t=3: && 0 0 0 0 1 0 2 4 0 1 0 1 0 0 0 2 1 0 2 4 0 0 0 1 1 1 0 4 4 2 2 0 0 0 0 0 0 0 0 0 0 0 0 0 0 0	0 0 \\
	t=4: && 0 0 0 0 0 1 0 2 4 0 1 0 1 0 0 0 0 2 1 2 4 0 0 0 0 0 1 1 1 0 0 4 4 2 2 0 0 0 0 0 0 0 0 0 0 0 0 0\\
	t=5: && 0 0 0 0 0 0 1 0 2 4 0 1 0 1 0 0 0 0 0 1 2 4 2 0 0 0 0 0 0 1 1 1 0 0 0 4 4 2 2 0 0 0 0 0 0 0 0 0\\
	t=6: && 0 0 0 0 0 0 0 1 0 2 4 0 1 0 1 0 0 0 0 0 1 2 0 4 2 0 0 0 0 0 0 0 1 1 1 0 0 0 0 4 4 2 2 0 0 0 0 0
\end{eqnarray*}

Observe that if there is a string of length $k$ consecutive balls of 
non-increasing (to the right) colors, then each of the balls moves at least distance $k$ in one iteration of the update rule. Suppose there is a such sequence of length $k_{1}$, followed by a sequence of $m$ 0's, and then another sequence of length $k_{2}$ of non-increasing colors. Also note that if $k_{1}\le m \le k_{2}$, then the two sequences do not interact in one iteration and the spacing becomes $m+(k_{2}-k_{1})\ge m$. On the other hand, if $k_{1}>k_{2}$, then the longer sequence on the left eventually catches up the shorter one and interfere. Since there are only finitely many balls in the initial configuration $X_{0}$, it is not hard to observe that after a finite number of iterations, the system decomposes into a disjoint non-interacting sequence of balls of non-increasing colors, whose lengths are non-decreasing from left to right. Each of such non-interacting sequence of balls is called a \textit{soliton}, and such decomposition is called a \textit{soliton decomposition} of the system $(X_{t})_{t\ge 0}$.

Given a basic $\kappa$-color BBS configuration $X_{0}:\mathbb{N}\rightarrow  \{0,1,\cdots,\kappa \}$, its soliton decomposition may be encoded in a Young tableau whose $j^{\text{th}}$ column corresponds to $j^{\text{th}}$ longest soliton. For instance, below is the Young tableau corresponding to the soliton decomposition of the $t=6$  instance of the 4-color BBS given before:  
\begin{equation} 
\young(2122111421,214,41,4)
\vspace{0.1cm}
\end{equation}

For $\kappa=1$, such a Young tableau is filled with all 1's, so its shape contains all relevant conserved quantities. For $\kappa>1$, however,  solitons possess not only their length but also their internal degrees of freedom. It turns out that we need a $\kappa$-tuple of Young diagrams $(\mu^{(1)}, \mu^{(2)},\cdots, \mu^{(\kappa)})$ to fully describe all such conserved quantities, where $\mu^{(1)}$ encodes the lengths of solitons and the other `higher order' Young diagrams describe their internal degrees of freedom. 
They provide a proper label of {\em iso-level sets} of the basic $\kappa$-color BBS \cite{kuniba2006crystal}. For our running example of 4-color BBS, the four invariant Young diagrams are   
\begin{equation}\label{eq:invariant_YD_ex}
{\normalsize\mu^{(1)}={\tiny\yng(10,3,2,1)},\quad 
\mu^{(2)}={\tiny\yng(5,2,1,1)},\quad \mu^{(3)}= \tiny{\yng(2,1,1)}, \quad 
\mu^{(4)}={\tiny\yng(1,1,1,1)}}.
\end{equation} 

In this paper, we study the limit shape of such invariant Young diagrams of the basic $\kappa$-color BBS when the initial configuration is randomly chosen. Fix a probability distribution $\mathbf{p}=(p_{0},p_{1},\cdots,p_{\kappa})$ on $\{ 0,1,\cdots,\kappa \}$ and let $X:\mathbb{N}\rightarrow \{0,1,\cdots,\kappa \}$ be a random $\kappa$-color BBS configuration where $\mathbb{P}(X^{\mathbf{p}}(x)=i)=p_{i}$ independently for all $x\in \mathbb{N}$ and $0\le i \le \kappa$. For each integer $n\ge 1$, denote
\begin{equation}\label{eq:def_iid_config}
X^{n,\mathbf{p}}(x) = X^{\mathbf{p}}(x)\cdot \mathbf{1}(1\le x \le n),
\end{equation} 
where $\mathbf{1}(A)$ denotes the indicator function of event $A$. A given basic $\kappa$-color BBS configuration $X_{0}:\mathbb{N}\rightarrow  \{0,1,\cdots,\kappa \}$ is said to be a \textit{highest state} if for all $n\ge 1$,
\begin{equation}\label{eq:highest}
\#( \text{balls of color $i$ in $X_{0}$ over $[1,n]$} ) \ge  \#( \text{balls of color $i+1$ in $X_{0}$ over $[1,n]$} ) \qquad \text{$\forall 0\le i < \kappa$},
\end{equation}
where we interpret balls of color 0 as empty boxes.

The following is a one-parameter specialization of the general
result, Theorem \ref{thm:SLLN_rows}, stated in the next section.

\begin{corollary}\label{cor:principal_specialization_kappa1}
	Consider the basic $\kappa$-color BBS initialized at $X^{n,\mathbf{p}}$, where $p_{i}\propto q^{i}$ for all $0\le i \le \kappa$ for some parameter $0<q<1$. Let $\rho_{i}^{(a)}(X^{n,\mathbf{p}})$ denote the $i^{\text{th}}$ row length of the $a$th invariant Young diagram $\mu^{(a)}$ of the BBS started at $X^{n,\mathbf{p}}$. 
	\begin{description}
		\item[(i)] Almost surely as $n\rightarrow \infty$, 
		\begin{align}\label{eq:p_specialization}
		n^{-1}\rho_{i}^{(a)}(X^{n,\mathbf{p}}) \longrightarrow \frac{q^{i+a-1}(1-q)(1-q^a)(1-q^{\kappa+1-a})}
		{(1-q^{\kappa+1})(1-q^{i+a-1})(1-q^{i+a})}.
		\end{align}
		\item[(ii)] The same almost sure convergence in (i) holds for $n^{-1}\rho_{i}^{(a)}(Y^{n,\mathbf{p}})$, where $Y^{n,\mathbf{p}}$ is $X^{n,\mathbf{p}}$ conditioned on being highest state.
	\end{description}
\end{corollary}

%\iffalse
\begin{figure*}[h]
	\centering
	\includegraphics[width=15cm,keepaspectratio]{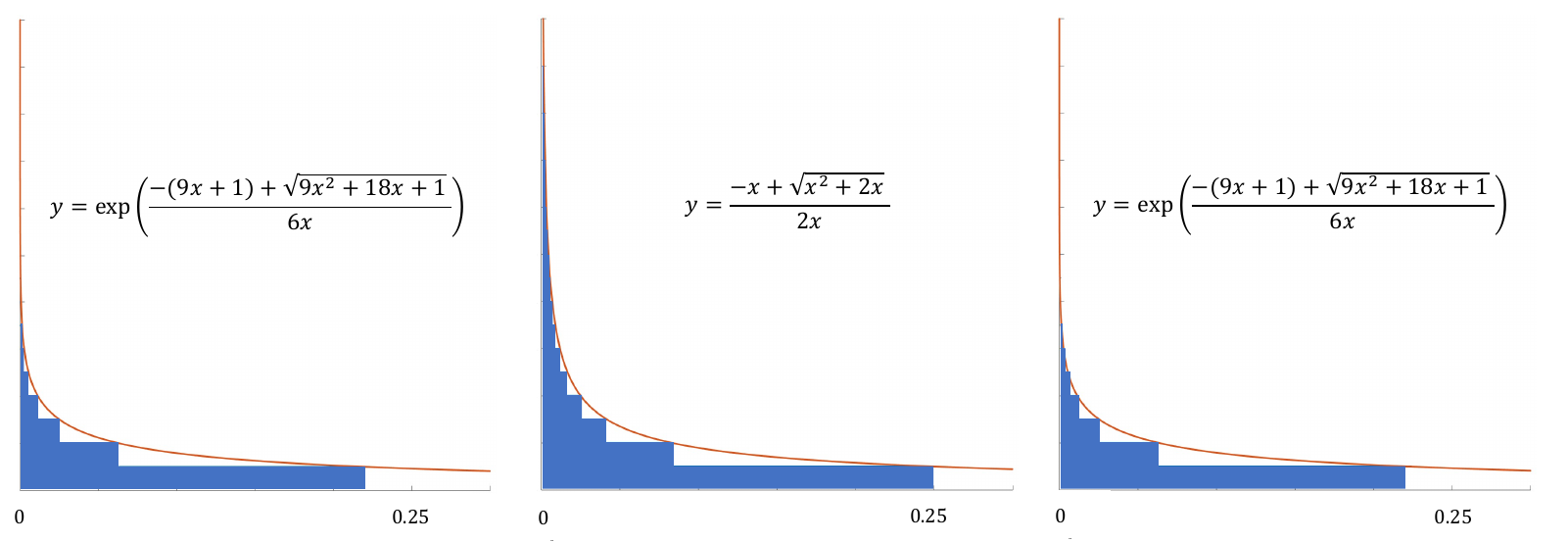}
			\vspace{-0.3cm}
	\caption{ Vertical flip of the invariant Young diagrams $\mu^{(1)}$ corresponding to the 1-color BBS of system size $n=500000$ with ball density $\mathbf{p}=(p_{0},p_{1})$ with $p_{1}=1/3$ (left), $p_{1}=1/2$ (middle), and $p_{1}=2/3$ (right). The equations for the limiting curve $y=y(x)$ are obtained by 
	$x = \eta^{(1)}_y$ from \eqref{eq:thm1_eq3}. The limiting curves for $p_{1}=1/3$ and $p_{1}=2/3$ are the same due to the `row duality' (see Remark \ref{remark:row_duality}).  
	}
	\label{fig:YD_duality}
\end{figure*}
%\fi

\subsection{Background}

The basic $\kappa$-color BBS was introduced in \cite{takahashi1993some}, generalizing the original $\kappa=1$ BBS first invented in 1990 \cite{takahashi1990soliton}. In the most general form of the BBS, each site accommodates a semistandard tableau of rectangular shape with letters from  $\{0,1,\cdots,\kappa \}$ and the time evolution is defined by successive application of the combinatorial $R$ (cf. \cite{fukuda2000energy, hatayama2001, kuniba2006crystal,inoue2012integrable}).  The basic $\kappa$-color BBS treated in this paper corresponds to the case where the tableau shape is the single box. By now, numerous aspects of BBS have been widely investigated in relation to quantum groups, crystal base theory (theory of quantum groups at $q=0$), solvable lattice models, Bethe ansatz, soliton equations, ultradiscretization, tropical geometry and so forth (see, e.g., the review article \cite{inoue2012integrable}).

In the recent years, BBS with random initial configuration has gained a considerable attention in the probability community, aiming to answer the following two central questions: 1) If the random initial configuration is one-sided, what is the limiting shape of the invariant random Young diagram as the system size tend to infinity? 2) If one considers the two-sided BBS (where the initial configuration is bi-directional array of balls), what are the two-sided random initial configurations that are invariant under the BBS dynamics? These questions have been addressed for the basic $1$-color BBS \cite{levine2017phase, ferrari2018soliton, ferrari2018bbs, croydon2018dynamics}. 
In the present work, we rigorously establish the row scaling limit of the invariant Young diagrams for the basic $\kappa$-color BBS for arbitrary $\kappa\ge 1$.  

There are two important works which are strongly related to this paper. In \cite{levine2017phase}, Levine, Lyu, and Pike studied various soliton statistics of the basic $1$-color BBS when the system is initialized according to a Bernoulli product measure with ball density $p$ on $n$ boxes. One of their main result shows that for any $p\in (0,1)$ and integer $i\ge 1$, the $i^{\text{th}}$ row length of the invariant Young diagram $\mu^{(1)}$ scales asymptotically linearly in $n$. This suggests that if we scale the Young diagram horizontally by $1/n$, then the sequence of random Young diagrams should converge to some limiting shape. Their analysis is based on some simple geometric mappings from the associated simple random walks to the invariant Young diagrams, which enables robust analysis of the scaling limit of the invariant Young diagram. However, this connection is not apparent in the general $\kappa\ge 1$ case. 

In a recent publication \cite{kuniba2018randomized}, Kuniba, Lyu and Okado developed and connected two parallel theories to analyze scaling limits of the invariant Young diagrams in the most general BBS, where each site is occupied with a semistandard tableau of rectangular shape with fillings $\{0,1,\cdots, \kappa\}$. Namely, when the initial configuration is given by a product measure, then one can obtain the `expected shape' of the invariant Young diagrams by a Markov chain method. On the other hand, if the initial configuration is conditioned on being the highest states (defined analogously to \eqref{eq:highest} in the general case), then one can analyze the `equilibrium shape' of the Young diagrams by the method of Thermodynamic Bethe Ansatz (TBA). Their main result \cite[Thm. 5.1]{kuniba2018randomized} states that the expected shape obtained from the Markov chain method coincides with the equilibrium shape obtained from the TBA method.

The main contribution of this paper shows that, for the basic $\kappa$-color BBS, the rescaled invariant Young diagrams converge almost surely to the equilibrium shapes as the system size tends to infinity. We also give a probabilistic answer to why the two different methods assuming different initial configuration produce the same scaling limit of the invariant Young diagrams. Roughly speaking, this is because the row lengths are exponentially concentrated around their mean due to a large deviations principle for Markov additive functionals, and the probability of $X^{n,\mathbf{p}}$ being highest decays only polynomially in $n$ (provided  $p_{0}\ge p_{1}\ge \cdots \ge p_{\kappa}$) due to a generalized ballot theorem. Hence one can push the same ergodic theorem from the unconditioned to the conditioned initial measure. This gives a rigorous proof of the `asymptotic equivalence' between the methods of Markov chains and TBA, which was first observed in \cite{kuniba2018randomized}.

\subsection{Notation}

We denote by $\mathbb{Z}$ the set of integers, and by $\mathbb{N}$ and $\mathbb{N}_{0}$ the set of positive and nonnegative integers, respectively.  When $a,b$ are integers, we denote by $[a,b]$ the integer interval $\{ k\in \mathbb{Z}\,:\, a\le k \le b \}$. For a probability distribution $\pi$ on a countable sample space $\Omega$ and a functional $f:\Omega\rightarrow \mathbb{R}$, we let $\EE_{\pi}[f]=\sum_{\omega\in \Omega} f(\omega)\pi(\omega)$ denote the expectation of $f$ with respect to $\pi$.

\subsection{Organization} 

In Section \ref{section:main_results}, we state our main results, Theorems \ref{thm:SLLN_rows} and \ref{thm:row_concentration}. In Section \ref{section:combinatorial_R}, we give a brief introduction to the combinatorial $R$ and define the invariant Young diagrams associated to a basic $\kappa$-color BBS configuration through the associated energy matrix. In Section \ref{section:algorithm_irreducibility}, we introduce the carrier processes and show that they form irreducible and stationary Markov chains (Theorem \ref{thm:stationary_measure} and Lemma \ref{lemma:carrier_irreducibility}). In Section \ref{section:proof of main results}, we prove limit theorems for the rows: (SLLN) Theorem  \ref{thm:SLLN_rows} (i), (LDP) Theorem \ref{thm:row_concentration}, and (FCLT and persistence) Theorem \ref{thm:FCLT_persistence}. 
In Section \ref{section:highest_states}, we study the effect of conditioning the initial measure on the highest states and prove Theorem \ref{thm:SLLN_rows} (ii). In the last Section \ref{section:TBA}, we give an alternate treatment of the basic $\kappa$-color BBS conditioned on highest states by the TBA analysis. 
A more general result in this direction has been obtained in \cite{kuniba2018randomized}. The case treated in this paper corresponds to its one parameter specialization for the basic BBS. 
However it deserves an independent report due to further results on the exact factorized 
scaling form of the vacancy (defined in Subsection \ref{sb:tatki}) 
and the column multiplicity in Proposition \ref{pr:res}.

\section{Main results}
\label{section:main_results}

We state our main results in this section. Let $\mathbf{p}=(p_{0},p_{1},\cdots,p_{\kappa})$ denote a probability distribution on $\{ 0,1,\cdots,\kappa \}$ such that $p_{i}>0$ for all $0\le i\le \kappa$. Let $X^{\mathbf{p}}$ and $X^{n,\mathbf{p}}$ be the random $\kappa$-color BBS configurations defined at \eqref{eq:def_iid_config} and above. We call $X^{\mathbf{p}}$ the \textit{i.i.d. basic $\kappa$-color BBS configuration with density $\mathbf{p}$}, where i.i.d. stands for `independent and identically distributed'.

Let $\lambda_{1}\ge \cdots \ge \lambda_{\kappa+1}\ge 0$ be nonnegative integers. We denote the Young diagram $\lambda$ of row lengths $\lambda_{1}\ge \cdots \ge \lambda_{\kappa+1}$ (allowing rows of length zero) as $\lambda=(\lambda_{1},\cdots, \lambda_{\kappa+1})$. Let $(c^{a})$ denote the $(a\times c)$ Young diagram $(c,c,\cdots,c)$. The \textit{Schur polynomial} of $\lambda$ for parameters $w_{1},\cdots,w_{\kappa+1}$, which is denoted by $s_{\lambda}(w_{1},\cdots,w_{\kappa+1})$, is defined by 
\begin{align}\label{qf}
s_{\lambda}(w_{1},\cdots,w_{\kappa+1}) = \frac{\det\left(w_i^{\lambda_{j}+\kappa+1-j}\right)_{i,j=1}^{\kappa+1}}
{\det\left(w_i^{\kappa+1-j}\right)_{i,j=1}^{\kappa+1}}.
\end{align}

In our first main result, Theorem \ref{thm:SLLN_rows}, we show that the invariant Young diagrams of the i.i.d. configuration $X^{n,\mathbf{p}}$ converge almost surely to some limiting shape, if we scale their rows by $n^{-1}$. Moreover, we also show that the same almost sure scaling limit still holds if we condition the initial configuration to be highest. 

\begin{theorem}[SLLN for rows]\label{thm:SLLN_rows}
	Consider the basic $\kappa$-color BBS initialized at $X^{n,\mathbf{p}}$. Let $\rho_{i}^{(a)}$ denote the $i^{\text{th}}$ row length of the $a$th invariant Young diagram $\mu^{(a)}$. 
	\begin{description}
		\item[(i)] For each $i\ge 1$ and $1\le a \le \kappa$, 
		almost surely as $n\rightarrow \infty$,    
		\begin{equation}\label{eq:thm1_eq1}
		n^{-1}\rho_{i}^{(a)}(X^{n,\mathbf{p}}) \rightarrow \eta^{(a)}_i
		= \eta^{(a)}_i(\kappa, \mathbf{p}) 
		:= \frac{s_{((i-1)^{a-1})}(p_{0},\cdots,p_{\kappa})\cdot s_{(i^{a+1})}(p_{0},\cdots,p_{\kappa})}{s_{(i^{a})}(p_{0},\cdots,p_{\kappa})\cdot s_{((i-1)^{a})}(p_{0},\cdots,p_{\kappa})} \in (0,1].
		\end{equation}
			\item[(ii)] Suppose $p_{0}\ge p_{1}\ge \cdots \ge p_{\kappa}$ and let $Y^{n,\mathbf{p}}$ denote $X^{n,\mathbf{p}}$ conditioned on being highest state. For each $i\ge 1$ and $1\le a \le \kappa$, 
		almost surely as $n\rightarrow \infty$,  
		\begin{equation}\label{eq:thm1_eq2}
		n^{-1}\rho_{i}^{(a)}(Y^{n,\mathbf{p}}) \rightarrow \eta^{(a)}_i \in (0,1],
		\end{equation}
		where $\eta^{(a)}_i$ is the same as \eqref{eq:thm1_eq1}.  
	\end{description}
\end{theorem}

We remark that the expression in \eqref{eq:thm1_eq1} involving Schur polynomials has already appeared in \cite[eq. (94)]{kuniba2018randomized} as the equilibrium of the rescaled row length for the i.i.d. configuration. Moreover, under the strong monotonicity assumption $p_{0}>p_{1}>\cdots>p_{\kappa}$, the same limiting shape was computed for the i.i.d. configuration conditional on being a highest state by Thermodynamic Bethe Ansatz (TBA). However, rigorous proof of almost sure convergence as stated in Theorem \ref{thm:SLLN_rows}, especially with the weak monotonicity assumption $p_{0}\ge p_{1}\ge \cdots \ge p_{\kappa}$, is a new contribution of the present paper.

\begin{remark}[Row duality]\label{remark:row_duality}
	Note that if $\kappa=a=1$ and $q:=p_{1}/p_{0}$, then (\ref{eq:thm1_eq1}) reads
	\begin{equation}\label{eq:thm1_eq3}
	\eta^{(1)}_i = \begin{cases}
	\frac{(q-1)^{2}q^{i}}{(q+1)(q^{i}-1)(q^{i+1}-1)} 	 & \text{if $p_{1}\ne 1/2$}\\
	\frac{1}{2i(i+1)} & \text{otherwise}.
	\end{cases}
	\end{equation}
	See Example \ref{ex:energy_computation_single_color} for more details. Observe that the function $\eta^{(1)}_1(1,\mathbf{p})$ is invariant under the swapping transformation $\mathbf{p}=(p_{0},p_{1})\mapsto \mathbf{p}'=(p_{1},p_{0})$ (see Figure \ref{fig:YD_duality}). This gives the `row version' of the duality between the supercritical and subcritical BBS with ball densities $p$ and $1-p$, first observed in \cite{levine2017phase} (see, in particular, Lemma 7.1). 
\end{remark}

In our second main result, Theorem \ref{thm:row_concentration}, we establish a \textit{large deviations principle} (LDP) for the row lengths of the i.i.d. configuration $X^{n,\mathbf{p}}$. This implies that the rate of convergence in Theorem \ref{thm:SLLN_rows} (i) is exponential. 

\begin{theorem}[LDP for rows]\label{thm:row_concentration}
	Consider the basic $\kappa$-color BBS initialized at $X^{n,\mathbf{p}}$. For each $i\ge 1$ and $1\le a \le \kappa$, there exists a convex rate function $\Lambda^{*}$ with the following properties: 
	\begin{description}
		\item[(i)] For any Borel set $F\subseteq \mathbb{R}$,  
		\begin{align}\label{eq:cramer_thm}
		-\inf_{u\in \mathring{F}} \Lambda^{*}(u) &\le \liminf_{n\rightarrow \infty} \frac{1}{n} \log \mathbb{P}\left(n^{-1} \rho_{i}^{(a)}(X^{n,\mathbf{p}}) \in F  \right) \\
		&\le \limsup_{n\rightarrow \infty} \frac{1}{n} \log \mathbb{P}\left(n^{-1} \rho_{i}^{(a)}(X^{n,\mathbf{p}}) \in F  \right) \le -\inf_{u\in \bar{F}} \Lambda^{*}(u),
		\end{align}
		where $\mathring{F}$ and $\bar{F}$ denotes the interior and closure of $F$, respectively. 
		\vspace{0.1cm}
		\item[(ii)] Let $\eta^{(a)}_i$ be the quantity defined at \eqref{eq:thm1_eq1}. Then there exists a constant $\nu\in (\eta^{(a)}_i,1]$ such that  $\Lambda^{*}\in (0,\infty)$ on $[0,\nu]\setminus \{\eta^{(a)}_i\}$ . 
	\end{description}
\end{theorem}
The rate function $\Lambda^{*}$ in the above theorem is explicitly given as the Legendre transform of the logarithmic principal eigenvalue of the exponentially weighted transition matrix of an underlying Markov chain called carrier process. See \eqref{eg:def_legendre} and also Examples \ref{ex:3-color_row1} and \ref{example2}.

\vspace{0.3cm}
\section{The combinatorial $R$ and the energy matrix}
\label{section:combinatorial_R}

The basic $\kappa$-color BBS has another description in terms of the crystal base theory \cite{kashiwara1991crystal} and the combinatorial $R$. While the definition of time evolution of BBS given in (\ref{eq:BBS_def_nonlocal}) involves the nonlocal movements of balls, such a crystal theoretical formulation allows a fully local description in terms of the accompanying `carrier process', which was first introduced in \cite{takahashi1997box} for the original $\kappa=1$ BBS. In the multicolor case, a similar carrier version of the time evolution can be defined, and the action of the carrier on each box is given by the combinatorial $R$. We give a brief introduction in this section, and refer to \cite{inoue2012integrable, lam2014rigged} for more backgrounds and details. 

\subsection{Combinatorial $R$ and the local energy $H$}
\label{subsection:combinatorial_R_def}

Fix an integer parameter $\kappa\ge 1$. For each integer $1\le m \le \kappa$ and $n\ge 1$, let $B^{(m)}_{n}=B^{(m)}_{n}(\kappa)$ be the set of all semistandard tableaux of rectangular shape $(m \times n)$ with letters from $\{0,1,\cdots, \kappa\}$. Equivalently, it is the set of all $(m \times n)$ matrices whose entries are from $ \{0,1,\cdots,\kappa \}$ and are weakly increasing on the rows and strictly increasing on columns. For a semistandard tableau $b$ and an integer $x\ge 0$, denote by $(b\leftarrow x)$ the tableau obtained by inserting $x$ into $b$ using Schensted row insertion. If $x$ is a word $x_{1}x_{2} \ldots x_{r}$ of nonnegative integers, then we can define $(b\leftarrow x)$ similarly by $(\cdots (((b\leftarrow x_{1})\leftarrow x_{2}) \cdots \leftarrow x_{r})$. For example, 
\begin{equation}\label{eq:ex_insertion}
\left( \young(12,34) \leftarrow 1\,3\,4 \right)  = \left( \young(11,24,3) \leftarrow 3\, 4 \right)  = \left( \young(113,24,3) \leftarrow 4 \right) = \young(1134,24,3).
\end{equation}
For each element $B\in B_{n}^{(m)}$, denote by $\textup{row}(B)$ the row word obtained by concatenating the rows of $B$ from bottom to top to the right. For example,
\begin{equation}
\textup{row}\left( \young(12,34) \right)= 3\, 4\, 1\, 2%\young(3412).
\end{equation}
Then, it is known that \cite{shimozono2002affine} there exists a unique map $R:B_{c}^{(a)} \times B_{s}^{(r)} \rightarrow B_{s}^{(r)} \times B_{c}^{(a)}$, $(C,B)\mapsto (B',C')$ satisfying 
\begin{equation}\label{eq:CR_factorization}
(B\leftarrow \textup{row}(C)) = (C'\leftarrow \textup{row}(B')).
\end{equation} 
We call this unique map the \textit{combinatorial $R$}. We denote by $R_{2}(C_{1},x_{1})\in B_{c}^{(a)}$ the second component of the image $R(C_{1},x_{1})$.

The associated \textit{local energy} is a function $H:B_{c}^{(a)} \times B_{s}^{(r)} \rightarrow \mathbb{N}_{0}$ where  
\begin{equation}\label{eq:def_H}
H(C,B) = \sum_{i>\max(a,r)} \text{length of the $i^{\text{th}}$ row of $(B\leftarrow \textup{row}(C)$} ).
\end{equation} 
For our running example in (\ref{eq:ex_insertion}), we have 
\begin{equation}
H\left( \young(12,34), \,  \young(134) \right)  = 1.
\end{equation}

Given $(C,B)$, an algorithm for finding 
$(B',C')$ satisfying (\ref{eq:CR_factorization}) 
is known \cite[p.55]{okado2007part} even though somewhat cumbersome. In the special case of the combinatorial $R$ acting on $B_{c}^{(a)}\times B_{s}^{(r)}$ with $a=r=1$ or $c=s=1$, a diagramatic computation rule was given in \cite{nakayashiki1997kostka}. In subsection \ref{subsection:pf_lemma_irreducibility}, we give a simple algorithm for the case $r=s=1$, which is sufficient to analyze the basic $\kappa$-color BBS.

\vspace{0.2cm}
\subsection{The carrier processes and the energy matrix}
\label{subsection:general_BBS}

At the beginning of Section \ref{section:combinatorial_R}, we mentioned that the basic $\kappa$-color BBS can be formulated in terms of the `carrier' and the combinatorial $R$. One of the advantages in such formulation is that the $\kappa$-tuple of invariant Young diagrams can be extracted by simply running carriers with different shapes over a given basic $\kappa$-color BBS configuration. In this subsection, we introduce a formal setup to state such observation.

Fix a basic $\kappa$-color BBS configuration $X$, which we may identify with a map $\mathbb{N}\rightarrow B_{1}^{(1)}=B_{1}^{(1)}(\kappa)$. Fix integers $(a,c)\in [1,\kappa]\times \mathbb{N}$. Let $U_{c}^{(a)}\in B_{c}^{(a)}$ be the $(a\times c)$ tableau such that each box in its $i^{\text{th}}$ row is filled with $i-1\in \{0,1,\cdots,\kappa \}$. We define two maps $\Gamma:\mathbb{N}_{0}\rightarrow B_{c}^{(a)}$ and $X':\mathbb{N}\rightarrow B_{1}^{(1)}$ as follows. First let $\Gamma(0)=U_{c}^{(a)}$ and then recursively for each $x\ge 1$, we require
\begin{equation}
R(\Gamma(x-1), X(x)) =  (X'(x),\Gamma(x))\in B_{1}^{(1)}\times B_{c}^{(a)},
\end{equation}  
where $R$ is the combinatorial $R$ defined in the previous subsection.

In words, we run a $B_{c}^{(a)}$-carrier on the array $X$ starting with the initial tableau $U_{c}^{(a)}$. The carrier with a certain state $\in B_{c}^{(a)}$ acts on the next site with state $\in B_{1}^{(1)}$ according to the combinatorial $R$, producing a site state and new carrier state. Hence $\Gamma(x)$ may be regarded as the carrier state after scanning the sites in the interval $[1,x]$. In this sense, we call $\Gamma_{c}^{(a)}=(\Gamma(x))_{x\ge 0}$ the \textit{$B_{c}^{(a)}$-carrier process over} $X$ (see Figure \ref{fig:BBS_transfer}). The integer parameters $c$ and $a$ are called the \textit{capacity} and \textit{height} of the carrier, respectively. 

%\iffalse
\begin{figure*}[h]
	\centering
	\vspace{-0.1cm}
	\includegraphics[width=1 \linewidth]{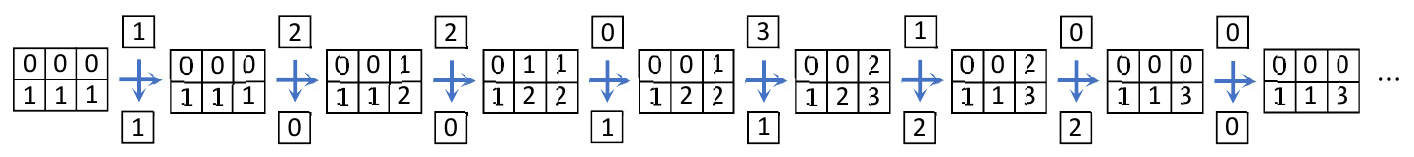}
	\vspace{-0.5cm}
	\caption{ Initial basic 3-color BBS configuration $X=(12203100\cdots)$ (top row) associated $B_{3}^{(2)}$-carrier process $\Gamma_{3}^{(2)}$ (middle row), and new basic 3-color BBS configuration $X'=(10011220\cdots)$ (bottom row). At each crossing in the diagram, a carrier state $\Gamma(x-1)$ and a box state $X(x)$ are mapped to a new box state $X'(x)$ and a new carrier state $\Gamma(x)$ by the combinatorial $R$. 
	}
	\label{fig:BBS_transfer}
\end{figure*}
%\fi

An important fact is that if $a=1$ and $c\ge 1$ is large enough relative to the total number of balls in $X$, then the induced update map $\mathcal{T}_{c}^{(a)}:X\mapsto X'$ in fact agrees with the time evolution of the basic $\kappa$-color BBS given in (\ref{eq:BBS_def_nonlocal}) (see. e.g., \cite{hatayama2001factorization}). Hence, this gives an alternate characterization of the basic $\kappa$-color BBS defined in the introduction\footnote{ The update map $\mathcal{T}_{c}^{(a)}$ may not preserve number of balls when $a>1$, as seen in Figure \ref{fig:BBS_transfer}. In order to make it a time evolution that preserves number of balls, we need to introduce `barriers' at the right tail of the state space. See \cite[Subsection 2.2]{kuniba2018randomized}.}.

The major advantage of the above construction of the basic $\kappa$-color BBS 
is that the conserved quantities are obtained 
by simply running carriers of different capacities and heights. Namely, for a given basic $\kappa$-color BBS configuration $X$ and for each integers $c\ge 1$ and $1\le a \le \kappa$, we define its \textit{(row transfer matrix) energy}
\begin{equation}\label{eq:def_row_transfer_energy}
E_{c}^{(a)}(X) = \sum_{x=0}^{\infty} H(\Gamma(x), X(x+1)),
\end{equation}
where $\Gamma_{c}^{(a)} = (\Gamma(x))_{x\ge 0}$ is the $B_{c}^{(a)}$-carrier process over $X$, and $H$ is the local energy function introduced in the previous subsection. One can see that (\ref{eq:def_row_transfer_energy}) is convergent from the definition (\ref{eq:def_H}). These quantities can be collected as a $(\infty\times \kappa)$ integer matrix $E(X)$ whose $(c,a)$ entry is $E_{c}^{(a)}(X)$, which we call the \textit{energy matrix} of $X$. For the running example of 4-color BBS given in the introduction, we have 
\begin{equation}\label{eq:ex_energy_mx}
X = (1 1 2 1 4 0 1 0 1 2 1 4 2 0 4 4 2 0 1 2 0 0 0 \cdots) 
\mapsto 
E(X) = \begin{bmatrix} 
10  &   5  &   2  &   1 \\
13  &   7  &   3  &   2 \\
15  &   8  &   4  &   3 \\
16  &   9  &   4  &   4 \\ 
16  &   9  &   4  &   4 \\
\vdots & \vdots & \vdots & \vdots \\
\end{bmatrix}.
\end{equation}

The invariants corresponding to the first column of the energy matrix were introduced in \cite{fukuda2000energy}. Based on the successive application of the Yang-Baxter equation for the combinatorial $R$, the authors also established its time invariance. A similar method shows the time invariance of the entire energy matrix. Namely, for any basic $\kappa$-color BBS trajectory $(X_{t})_{t\ge 0}$, we have 
\begin{equation}\label{eq:time_invariance2}
E( X_{t} ) \equiv  E( X_{0} ) \quad \forall t\ge 0.
\end{equation}   

Given the invariance of the energy matrix, we construct a $\kappa$-tuple  $(\mu^{(1)}(X_{t}),\cdots,\mu^{(\kappa)}(X_{t}))$ of invariant Young diagrams by setting the length of the $i^{\text{th}}$ row of the $a^{\text{th}}$ Young diagram $\mu^{(a)}(X_{t})$, which we denote by $\rho_{i}^{(a)}(X_{t})$, through the following relation
\begin{equation}\label{eq:def_YDs}
\rho_{1}^{(a)}(X_{t})+\cdots+\rho_{i}^{(a)}(X_{t}) = E_{i}^{(a)}(X_{t}) \quad \text{$\forall$ $i\ge 1$ and $1\le a \le \kappa$}.
\end{equation} 
We note that there is another construction of these invariant Young diagrams using the Kerov-Kirillov-Reshetikhin bijection \cite{kerov1988combinatorics}. The equivalence between the two constructions is proven in \cite{sakamoto2009kirillov}. Hence (\ref{eq:def_YDs}) does indeed define valid Young diagrams.

We remark that `physical' meaning of the energy matrix can be given in terms of the soliton decomposition. For instance, let $t\ge 1$ be large enough so that all solitons of different lengths are separated with enough spacing between them. If we run the $B_{1}^{(1)}$-carrier over $X_{t}$, then the summation (\ref{eq:def_row_transfer_energy}) for $E_{1}^{(1)}(X_{t})$ picks up the left boundaries of distinct solitons, so the corresponding energy equals the total number of solitons, as seen from (\ref{eq:def_YDs}) with $i=a=1$. More generally, $\rho_{i}^{(1)}(X_{t})$ equals the number of solitons of length $\ge i$ (see \cite{fukuda2000energy} for a proof). See \cite{kuniba2006crystal} for soliton interpretations of the higher order Young diagrams.

\vspace{0.2cm}

\section{Irreducibility and stationary of the carrier processes}
\label{section:algorithm_irreducibility}

\subsection{Carrier processes and joint irreducibility}

Let $\mathbf{p}=(p_{0},p_{1},\cdots,p_{\kappa})$ be a probability distribution on $\{ 0,1,\cdots,\kappa \}$ such that $p_{i}>0$ for all $0\le i\le \kappa$. Let $X^{\mathbf{p}}$ and $X^{n,\mathbf{p}}$ be the random $\kappa$-color BBS configurations defined in Subsection \ref{subsection:intro1}. We call $X^{\mathbf{p}}$ the \textit{i.i.d. basic $\kappa$-color BBS configuration with density $\mathbf{p}$}, where i.i.d. stands for `independent and identically distributed'. If we run the $B_{c}^{(a)}$-carrier over $X^{\mathbf{p}}$, then the associated carrier process $\Gamma_{c}^{(a)}=(\Gamma(x))_{x\ge 0}$ becomes a stochastic process. In fact, since its evolution is determined by the combinatorial $R$ and the i.i.d. input $X^{\mathbf{p}}$, the carrier process becomes a Markov chain on the finite state space $B_{c}^{(a)}=B_{c}^{(a)}(\kappa)$. We call this Markov chain as the \textit{$B_{c}^{(a)}$-carrier process over $X^{\mathbf{p}}$}. 

\begin{example}\label{example1}
	Let $\kappa=2$ and ball density $\mathbf{p}=(p_{0},p_{1},p_{2})$, where $p_{i}>0$ for all $0\le i \le 2$. Consider the carrier process $\Gamma_{5}^{(1)}$ over $X^{\mathbf{p}}$.  We may identify a tableau $C\in B_{5}^{(1)}(2)$ with an integer vector $(c_{1},c_{2})\in \mathbb{N}_{0}^{2}$ such that $c_{1}+c_{2}\le 5$, where $c_{i} = \#(\text{balls of color $i\in C$})$. For example, 
	\begin{equation}
	\young(00122) \simeq (1,2).
	\end{equation}

	Then by invoking the algorithm for computing the combinatorial $R$ given above, we can obtain the state space diagram of this chain as shown in Figure \ref{fig:MC_diagram}. Since $p_{i}>0$ for all $0\le i \le 2$, clearly the chain is irreducible. Since the state space is finite,  there must be a unique stationary distribution, which we may denote by $\pi=\pi_{5}^{(1)}$ in this example. Observe that this stationary distribution is characterized by the following detailed balance equations:
	\begin{equation}\label{eq:balance_2color}
	\begin{cases}
	\pi(a,b) =   \pi(a,b+1)p_{0} +\pi(a-1,b)p_{1} +  \pi(a+1,b-1) p_{2} & \text{for $1\le a,b$ and $a+b\le 4$} \\
	\pi(a,0) =   \pi(a+1,0)p_{0} +\pi(a-1,0)p_{1} +  \pi(a,1)p_{0}   & \text{for $1\le a\le 4$} \\
	\pi(0,b) =  \pi(0,b+1) p_{0} +  \pi(1,b-1)p_{2} +\pi(0,b-1)p_{2}   & \text{for $1\le b\le 4$}\\
	\pi(a,b) =  \pi(a-1,b) p_{1} +  \pi(a+1,b-1)p_{2} +\pi(a-1,b+1)p_{1}   & \text{for $1\le a,b$ and $a+b=5$} \\
	\pi(0,0) = \pi(0,0)p_{0} + \pi(1,0)p_{0} + \pi(0,1) p_{0} & \\
	\pi(5,0) = \pi(4,0)p_{1} + \pi(5,0)p_{1} + \pi(4,1)p_{1} & \\
	\pi(0,5) = \pi(0,4)p_{2} + \pi(1,4)p_{2} + \pi(0,5)p_{2} &
	\end{cases}
	\end{equation}   
	
		\begin{figure*}[h]
		\centering
		\includegraphics[width=0.47 \linewidth]{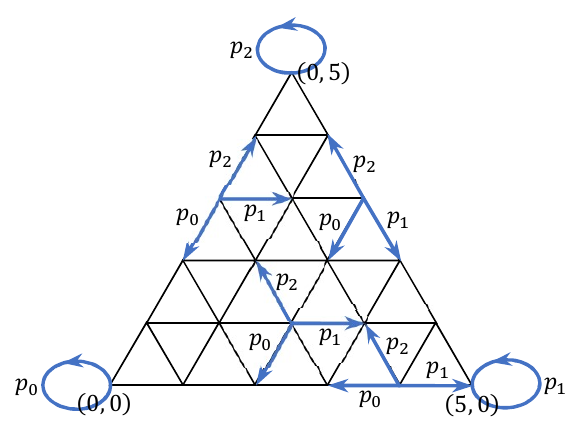}
		\caption{ State space diagram for the carrier process $\Gamma^{(1)}_{5}$ on the state space $\{ (c_{1},c_{2})\in \mathbb{N}_{0}^{2}\,|\, c_{1}+c_{2}\le 5 \}$. Each node has coordinate $(c_{1},c_{2})$, where $c_{i}$ is the number of balls of color $i$ currently in the carrier. The transition kernel on each of the three sides and the interior of the state space (viewed as a triangular lattice) is constant, and it is illustrated in blue arrows with probabilities $p_{i}$'s at some of the representative nodes. 
		}
		\label{fig:MC_diagram}
	\end{figure*}

	It is straightforward to check that the stationary distribution $\pi$ is given by 
	\begin{equation}\label{eq:stationary_measure_product_formula_ex}
	\pi(a,b) = \frac{1}{Z_{5}^{(1)}} p_{0}^{5-a-b}p_{1}^{a}p_{2}^{b},
	\end{equation}
	where the partition function $Z$ can be written as 
	\begin{equation}
	Z_{5}^{(1)} = \sum_{\substack{a,b,c\ge 0 \\ a+b+c=5}} p_{0}^{a}p_{1}^{b}p_{2}^{c}.
	\end{equation}
	$\hfill \blacktriangle$ 
\end{example}

The grounding result in this paper is that the carrier processes over $X^{\mathbf{p}}$ is an irreducible Markov chain with a unique stationary distribution, which is given by a simple product form. Furthermore, its partition function (normalization constant) becomes a Schur polynomial, which was defined at \eqref{qf}.
\begin{theorem}\label{thm:stationary_measure}
	The $B_{c}^{(a)}$-carrier process over $X^{\mathbf{p}}$ is an irreducible Markov chain with a unique stationary distribution $\pi_{c}^{(a)}$. Furthermore, $\pi_{c}^{(a)}$ is given by  
	\begin{equation}\label{eq:stationary_measure_product_formula}
	\pi_{c}^{(a)}(C) = \frac{1}{Z_{c}^{(a)}} \prod_{i=0}^{\kappa} p_{i}^{m_{i}(C)},
	\end{equation}
	where $m_{i}(C)$ denotes the number of $i$'s in the semistandard tableau $C$ and the normalization constant $Z_{c}^{(a)}=Z_{c}^{(a)}(\kappa,\mathbf{p})$ is given by 
	\begin{align}\label{eq:partition_schur}
	Z_{c}^{(a)}(\kappa,\mathbf{p})  = s_{(c^{a})}(p_{0},p_{1},\cdots,p_{\kappa}).
	\end{align}
\end{theorem}

In particular, we have 
\begin{align}
Z_{1}^{(1)}=s_{(1)}(p_{0},p_{1},\cdots,p_{\kappa}) = 
p_{0}+p_{1}+\cdots+p_{\kappa}=1
\end{align}
and $\pi_{1}^{(1)}=\mathbf{p}$. We note that the stationarity of $\pi_{c}^{(a)}$ has been established in the more general BBS in \cite[Prop. 3.2]{kuniba2018randomized}, where each site is occupied by a semistandard tableau of rectangular shape, not necessarily by a tableau of shape $(1\times 1)$ (a single ball). We include a proof of stationarity below for self-containedness. 
On the other hand, the irreducibility of the carrier process has not been addressed in \cite{kuniba2018randomized}. For the irreducibility statement in Theoerem \ref{thm:stationary_measure}, we prove a stronger statement (Lemma \ref{lemma:carrier_irreducibility}) for the `joint carrier process'.  

More precisely, for fixed integers $c\ge 1$ and $1\le a \le \kappa$, we jointly evolve the $B_{c}^{(a)}$- and $B_{c+1}^{(a)}$-carrier processes $(\Gamma_{c}^{(a)}(t))_{t\ge 0}$ and $(\Gamma_{c+1}^{(a)}(t))_{t\ge 0}$ over $X^{\mathbf{p}}$, respectively. Then the pair $(\Gamma_{c}^{(a)}(t), \Gamma_{c+1}^{(a)}(t))$ defines a Markov chain on the state space $\Omega:=B_{c}^{(a)}\times B_{c+1}^{(a)}$ with initial state $(U_{c}^{(a)},U_{c+1}^{(a)})$. Note that from \eqref{eq:def_YDs}, we can write
\begin{align}
\rho_{c+1}^{(a)}(X^{n,\mathbf{p}}) &= E_{c+1}^{(a)}(X^{n,\mathbf{p}}) - E_{c}^{(a)}(X^{n,\mathbf{p}}) 
\label{bak}\\
&=\sum_{k=1}^{n} \left( H(\Gamma_{c+1}^{(a)}(k-1),X^{\mathbf{p}}(k)) -  H(\Gamma_{c}^{(a)}(k-1),X^{\mathbf{p}}(k)) \right). \label{eq:rho_MC_functional}
\end{align}
Hence it is natural to keep track of the two carrier processes $\Gamma_{c}^{(a)}(t)$ and $\Gamma_{c+1}^{(a)}(t)$ at the same time.

Notice that since these two carrier processes are coupled through the joint evolution over $X^{\mathbf{p}}$, the joint carrier process is not necessarily irreducible over the entire state space $\Omega$. In Lemma \ref{lemma:carrier_irreducibility} below, we show that the joint process is irreducible on a smaller state space $\Omega_{0}\subseteq \Omega$ containing the `ground state' $(U_{c}^{(a)}, U_{c+1}^{(a)})$. This will also be crucial in establishing the large deviations principle for the row length (Theorem \ref{thm:row_concentration}) in Subsection \ref{subsection:LDP}. For its statement, we say a pair $(C_{1},C_{2})\in B_{c}^{(a)}\times B_{c+1}^{(a)}$ is \textit{jointly reachable} if 
\begin{align}
\PP\left( \text{$(\Gamma_{c}^{(a)}(t), \Gamma_{c+1}^{(a)}(t)) = (C_{1},C_{2})$ for some $t\ge 0$} \right)>0.
\end{align}

\begin{lemma}\label{lemma:carrier_irreducibility}
	Fix $c\ge 1$, $1\le a \le \kappa$, and let $(\Gamma_{c}^{(a)}(t), \Gamma_{c+1}^{(a)}(t))_{t\ge 0}$ be the joint carrier process over $X^{\mathbf{p}}$. Denote by $\Omega_{0}\subseteq \Omega$ the set of all jointly reachable states. Then the followings hold. 
	\begin{description}
		\item[(i)] For each $C\in B_{c}^{(a)}$, there exists $C'\in B_{c+1}^{(a)}$ such that $(C,C')\in \Omega_{0}$.
		\vspace{0.1cm}
		
		\item[(ii)] 	$\PP\big( \text{$(\Gamma_{c}^{(a)}(t), \Gamma_{c+1}^{(a)}(t))$ visits $(U_{c}^{(a)}, U_{c+1}^{(a)})$ infinitely often}\big)=1.$

		\vspace{0.1cm}
		\item[(iii)] $(\Gamma_{c}^{(a)}(t), \Gamma_{c+1}^{(a)}(t))_{t\ge 0}$ is irreducible on $\Omega_{0}$. 
	\end{description}
\end{lemma}

The proof of Lemma \ref{lemma:carrier_irreducibility} is given in the following subsection. Below we prove Theorem \ref{thm:stationary_measure} assuming this lemma.

\begin{proof}[\textbf{Proof of Theorem \ref{thm:stationary_measure}}]
	Let $\mathtt{Y}_{t}=(\Gamma_{c}^{(a)}(t), \Gamma_{c+1}^{(a)}(t))_{t\ge 0}$ denote the joint carrier process over $X^{\mathbf{p}}$ and let $\Omega_{0}$ denote the set of all jointly reachable states as in Lemma \ref{lemma:carrier_irreducibility}. Note that we can obtain the $B_{c}^{(a)}$-carrier process over $X^{\mathbf{p}}$ by projecting the joint carrier process $\mathtt{Y}_{t}$ onto its first coordinate. Fix tableaux $C_{1},C_{2}\in B_{c}^{(a)}$. Then  by Lemma \ref{lemma:carrier_irreducibility} (i), there exists $C_{1}',C_{2}'\in B_{c+1}^{(a)}$ such that $(C_{i},C_{i}')\in \Omega_{0}$ for $i=1,2$. It follows that  
	\begin{align}
		\PP\big(\Gamma_{c}^{(a)}(t)=C_{1} \big) \ge \PP\big(\mathtt{Y}_{t}=(C_{1},C_{1}')\big) >0.
	\end{align}
	Hence by Lemma \ref{lemma:carrier_irreducibility} (iii), we have  
	\begin{align}
	&\PP\big(\text{$\Gamma_{c}^{(a)}(s)=C_{2}$ for some $s\ge t$} \,|\, \Gamma_{c}^{(a)}(t)=C_{1}\big) \\
	&\qquad \ge \PP\big( \text{$\mathtt{Y}_{s}=(C_{2},C_{2}')$ for some $s\ge t$} \,|\, \Gamma_{c}^{(a)}(t)=C_{1}\big)  \\
	&\qquad \ge \PP\big( \text{$\mathtt{Y}_{s}=(C_{2},C_{2}')$ for some $s\ge t$} ,\,  \Gamma_{c}^{(a)}(t)=C_{1}\big) \\
	&\qquad \ge \PP\big(\text{$\mathtt{Y}_{s}=(C_{2},C_{2}')$ for some $s\ge t$} ,\, \mathtt{Y}_{t}=(C_{1},C_{1}')\big) > 0.
	\end{align} 
	This shows that the carrier process $(\Gamma_{c}^{(a)}(t))_{t\ge 0}$ is irreducible on $B_{c}^{(a)}$.

	Next, we show $\pi_{c}^{(a)}$ is a stationary distribution for $(\Gamma_{c}^{(a)}(t))_{t\ge 0}$. For each semistandard tableau $T$ (not necessarily rectangular), denote $e^{\wt(T)} = \prod_{i=0}^{\kappa} p_{i}^{m_{i}(T)}$, where $m_{i}(T)$ denotes the total number of letter $i$ in $T$ as before. For any two rectangular tableaux $S,T$, we may denote $S\cdot T = (S\leftarrow \text{row}(T))$ (see (\ref{eq:CR_factorization}) and above). Since the total number of letter $i$ is preserved in each row insertion step, we have $m_{i}(S\cdot T) = m_{i}(S)+m_{i}(T)$. Now for any $(C,B)\in B_{c}^{(a)}\times B_{1}^{(1)}$ and $(B',C')\in B_{1}^{(1)}\times B_{c}^{(a)}$ such that $R(C,B) = (B',C')$, we have $C\cdot B = B'\cdot C'$ by definition so we have   
	\begin{align}
	e^{\wt(C)}e^{\wt(B)} = e^{\wt(C\cdot B)} = e^{\wt(B'\cdot C')} = e^{\wt(B')}e^{\wt(C')}. 
	\end{align} 
	Moreover, recall the combinatorial $R:B_{c}^{(a)}\times B_{1}^{(1)}\rightarrow B_{1}^{(1)}\times B_{c}^{(a)}$ is a bijection, so it also gives a bijection between $B_{1}^{(1)}\times \{ C'\}$ and its inverse image under $R$. Hence we have 
	\begin{align}
	\sum_{\substack{(C,B)\in B_{c}^{(a)}\times B_{1}^{(1)} \\ R_{2}(C,B)=C' }} e^{\wt(C)}e^{\wt(B)} & =  \sum_{\substack{(C,B)\in B_{c}^{(a)}\times B_{1}^{(1)} \\ R_{2}(C,B)=C' }} e^{\wt(B')}e^{\wt(C')} 
	= e^{\wt(C')} \sum_{B'\in B_{1}^{(1)}} e^{\wt(B')}.
	\end{align}
	Note that $e^{\wt(i)} = p_{i}$ for each $0\le i \le \kappa$, so the summation in the last expression equals 1. 
	Hence dividing both sides by the partition function $Z_{c}^{(a)}$ gives\footnote{Here
	and in what follows we will often write $i \in \{0,1,\ldots, \kappa\}$ to mean $\young(i) \in B^{(1)}_1$ and vice versa.} 
	\begin{align}
	\sum_{\substack{(C,i)\in B_{c}^{(a)}\times B_{1}^{(1)} \\ R_{2}(C,i)=C' }} \pi_{c}^{(a)}(C) p_{i}
	= \pi_{c}^{(a)}(C').
	\end{align}
	This shows that $\pi_{c}^{(a)}$ is a stationary distribution for the $B_{c}^{(a)}$-carrier process over $X^{\mathbf{p}}$. The uniqueness follows from irreducibility and the fact that the state space $B_{c}^{(a)}$ is finite. The formula (\ref{eq:partition_schur}) follows from an alternative definition of the Schur function as the generating sum of 
	$e^{\wt}$ over the semistandard tableaux.
\end{proof}

\subsection{An algorithm for combinatorial $R$ and proof of Lemma \ref{lemma:carrier_irreducibility}}
\label{subsection:pf_lemma_irreducibility}

In this subsection we prove Lemma \ref{lemma:carrier_irreducibility}. Recall that the combinatorial $R$  introduced in Section \ref{section:combinatorial_R}, which we denote by $R:B_{c}^{(a)}\times B_{1}^{(1)}\rightarrow B_{1}^{(1)}\times B_{c}^{(a)}$, is defined implicitly through the factorization condition \eqref{eq:CR_factorization}. This definition is sufficient to derive the stationary distribution \eqref{eq:stationary_measure_product_formula} for the carrier process, but it is not suited to show the joint irreducibility stated in Lemma \ref{lemma:carrier_irreducibility}.  

To overcome this difficulty, we give an explicit map $\tilde{R}:B_{c}^{(a)}\times B_{1}^{(1)}\rightarrow B_{1}^{(1)}\times B_{c}^{(a)}$ by a simple algorithm, and show that this alternative map agrees with the combinatorial $R$ (Proposition \ref{prop:LICI}). We remark that a similar but more involved algorithm to compute combinatorial $R$ in the general case is known \cite[p.55]{okado2007part}. Also, in the special case of the combinatorial $R$ acting on $B_{c}^{(a)}\times B_{s}^{(r)}$ with $a=r=1$ or $c=s=1$, a diagramatic computation rule was given in \cite{nakayashiki1997kostka}. For our purpose, we need an algorithm for general $c,a\ge 1$ and $s=r=1$.

Our algorithm consists of two essential steps, namely, the \textit{reverse bumping} and \textit{column insertion}. The former is the reverse of the usual Schensted row insertion $(T\leftarrow x)$ from the bottom to the top row (See \cite[p.8]{fulton1997young}.) The latter is the usual Schensted column insertion from the first to the last column. (See \cite[p.186]{fulton1997young}.) More precisely, for each $c,a\ge 1$, we define a map $\tilde{R}:B_{c}^{(a)}\times B_{1}^{(1)}\rightarrow B_{1}^{(1)}\times B_{c}^{(a)}$, $(T,x)\mapsto (y,S)$, as below. 
Let $z$ be the number inscribed in the bottom left box in the tableau $T$. 
\begin{description}
	\item{(i) (\textit{reverse bumping})} Case $x>z$. 
	Replace the rightmost element $x'$ such that $x'<x$ in the bottom row of $T$  by $x$. 
	Replace the rightmost element $x''$ such that $x''<x'$ in the row above 
	(row $(c-1)$) of $T$ by $x'$. Perform the similar replacements all the way to the top row. The resulting $(a\times c)$ tableau is $S$, and the last replaced letter from the top row is $y$. The algorithm then halts. 
	\vspace{0.1cm}
	\item{(ii) (\textit{column insertion})} Case $x\le z$. Replace the topmost element $x'$ such that $x'\ge x$ in the first column of $T$ by $x$. 
	Replace the topmost element $x''$ such that $x''\ge x'$ in the second column of $T$ by $x'$. Perform the similar replacements all the way to the last column.
	The resulting $(a\times c)$ tableau is $S$, and the last replaced letter from the rightmost column is $y$. The algorithm then halts. 
\end{description}

\begin{example}
	We give four instances of the map $\tilde{R}$. In this example we use $\Rightarrow$ and $\Downarrow$ to indicate the intermediate steps of the reverse bumping and column insertion for single row and column, respectively\footnote{
The usage of the symbols $\Rightarrow$ and $\Downarrow$ here is not standard.}.
	%\iffalse
	\begin{align}
	&{\small\left(\young(011,234),\, \young(3) \right) \mapsto \left( \raisebox{-0.61 em}{\text{$3 \Rightarrow$} } \young(011,234) \right) \mapsto  \left( \raisebox{0.61 em}{\text{$2 \Rightarrow$} } \young(011,334) \right) \mapsto \left( \young(1),\, \young(012,334) \right)}, \label{eq:ex_inverted_insertion1}\\
	&{\small\left(\young(011,234),\, \young(2) \right) \mapsto \begin{pmatrix} \begin{matrix} 2 \\ \Downarrow \end{matrix} \text{\hspace{0.68cm}} \\ \young(011,234) \end{pmatrix}  \mapsto 
		\begin{pmatrix} \begin{matrix} 2 \\ \Downarrow \end{matrix} \text{\hspace{0cm}} \\ \young(011,234) \end{pmatrix} \mapsto 
		\begin{pmatrix} \begin{matrix} 3\\ \Downarrow \end{matrix} \text{\hspace{-0.68cm}} \\ \young(011,224) \end{pmatrix} \mapsto 
		\left(\young(4),\,\young(011,223)\right)}, \label{eq:ex_inverted_insertion2} \\
	& {\small \left( \young(011,224,445), \, \young(3) \right) \mapsto 
		\begin{pmatrix} \begin{matrix} 3 \\ \Downarrow \end{matrix} \text{\hspace{0.68cm}} \\ \young(011,224,445) \end{pmatrix} \mapsto 
		\begin{pmatrix} \begin{matrix} 4 \\ \Downarrow \end{matrix} \text{\hspace{0cm}} \\ \young(011,224,345) \end{pmatrix} \mapsto 
		\begin{pmatrix} \begin{matrix} 4 \\ \Downarrow \end{matrix} \text{\hspace{-0.68cm}} \\ \young(011,224,345) \end{pmatrix} \mapsto 
		\left(\young(4),\, \young(011,224,345) \right)},
	\label{eq:ex_inverted_insertion3}\\
	&{\small	 \left( \young(011,224,445), \, \young(0) \right) \mapsto 
		\begin{pmatrix} \begin{matrix} 0 \\ \Downarrow \end{matrix} \text{\hspace{0.68cm}} \\ \young(011,224,445) \end{pmatrix} \mapsto 
		\begin{pmatrix} \begin{matrix} 0 \\ \Downarrow \end{matrix} \text{\hspace{0cm}} \\ \young(011,224,445) \end{pmatrix} \mapsto 
		\begin{pmatrix} \begin{matrix} 1 \\ \Downarrow \end{matrix} \text{\hspace{-0.68cm}} \\ \young(001,224,445) \end{pmatrix} \mapsto 
		\left( \young(1),\, \young(001,224,445) \right)}.  \label{eq:ex_inverted_insertion4}
	\end{align}
	%\fi
	Observe from example (\ref{eq:ex_inverted_insertion4}) that if $0$ is inserted, then the algorithm always halts by sliding the top row to the right by a single cell and the $(1,1)$ cell is filled with a new $0$. $\hfill\blacktriangle$
\end{example}

Now we show that our algorithm $\tilde{R}$ agrees with the combinatorial $R$ defined at \eqref{eq:CR_factorization}.

\begin{proposition} \label{prop:LICI}
	The map $\tilde{R}:B_{c}^{(a)}\times B_{1}^{(1)}\rightarrow B_{1}^{(1)}\times B_{c}^{(a)}$ defined above agrees with the combinatorial $R$ acting on $B_{c}^{(a)}\times B_{1}^{(1)}$.  Furthermore, the local energy function $H$ is given by
	\begin{equation}\label{eq:H_characterization_ascent}
	H(T,x) = \mathbf{1}( \text{$x>$ bottom left element in $T$} ).
	\end{equation}  
\end{proposition}

\begin{proof}
	Suppose $R: (T,x) \mapsto (y,S)$.
	Then by definition and duality between row and column insertion \cite[Appendix A.2]{fulton1997young}, we have 	
	\begin{align}\label{syxt}
	(S \leftarrow y) = (x \leftarrow \mathrm{row}(T)) = (x \rightarrow T),
	\end{align}
	where the first equality is due to (\ref{eq:CR_factorization}) and $\rightarrow$ denotes the 
	column insertion \cite[p186]{fulton1997young}. 
	The second equality follows from the equivalence of the two constructions of 
	product tableaux based on row and column insertions. 
	Compare the bottom formula on \cite[p11]{fulton1997young} and the top one on \cite[p187]{fulton1997young}.
	Let $z$ be the bottom left letter in the tableau $T$ as in the above algorithm (i) and (ii).
	(I) Suppose $x>z$.
	Then $(x\rightarrow T)$ is obtained by attaching a single box containing $x$ 
	below the first column of $T$.
	To find $(y,S)$ from such $(x\rightarrow T)$ by the postulate (\ref{syxt}) is done 
	exactly by the reverse bumping procedure in (i).
	(II) Suppose $x\le z$. Then $(x \rightarrow T)$ has the shape of a $(a\times c)$ tableau with a single box attached to the right of its first row. Let $S'$ be its left $(a\times c)$ part and $y'$ be the remaining single box. Since $y'$ the largest in the first row of $(x\rightarrow T)$, it follows that $(S'\leftarrow y')=(x\rightarrow T)$. By the uniqueness, it follows that $S'=S$ and $y'=y$. Lastly, from (\ref{eq:def_H}) we have $H(T,x)=1$ for case (I) and $H(T,x)=0$ for case (II), proving (\ref{eq:H_characterization_ascent}).
\end{proof}

Now we are ready to prove Lemma \ref{lemma:carrier_irreducibility}.

\begin{proof}[\textbf{Proof of Lemma \ref{lemma:carrier_irreducibility}}.]
	
	Denote by $R_{2}$ the second coordinate of combinatorial $R$ as defined at \eqref{eq:CR_factorization}. For a tableau $T \in B^{(a)}_c$ and its row word $\mathrm{row}(T) = t_1t_2\cdots t_{ac}$,
	we denote the reverse row word by $\mathrm{row}'(T) = t_{ac}\cdots t_2t_1$.
	For any tableau $C\in B_{c}^{(a)}$ 
	and a word $x_{1}x_{2}\cdots x_{r}$, $0\le x_{i}\le \kappa$, define 
	\begin{equation}
	R_{2}(C, x_{1}x_{2}\cdots x_{r}) := 	R_{2}(\cdots R_{2}(R_{2}(C,x_{1}),x_{2})\cdots ,x_{r}) \in B_{c}^{(a)}. 
	\end{equation}
	For given two tableaux $C,C'\in B_{c}^{(a)}$, say $C$ reaches $C'$ if
	\begin{equation}\label{r2cc}
	R_{2}(C, x_{1}x_{2}\cdots x_{k}) = C'
	\end{equation}
	for some $x_{i}\in \{0,1, \cdots, \kappa \}$, $1\le i \le r$. Lastly, let $V_{c}^{(a)}$ be the `lowest tableau', whose entries in row $a-i$ are $(\kappa-i)$ for all $0\le i < a$.

	Using the realization of $R$ by the reverse bumping and the column insertion (Proposition \ref{prop:LICI}), it is not hard to see
	\begin{align}\label{eq:column_bump_pf}
	R_2(C,\mathrm{row}'(V^{(a)}_c)) = V^{(a)}_c, \qquad
	R_2(V^{(a)}_c, \mathrm{row}'(C)) = C
	\end{align}
	 for any $C \in B^{(a)}_c$. 
	See below for an illustration when $\kappa=5$, where we show $\Rightarrow$ and $\Downarrow$ only for the first step of successive reverse bumping and column insertion for inserting each letter.
	%\iffalse
	{\small 
		\begin{align}
		&\begin{pmatrix} \begin{matrix} 3 \\ \Downarrow \end{matrix} \text{\hspace{0.34cm}} \\ \young(02,23,34) \end{pmatrix} \mapsto 
		\begin{pmatrix} \begin{matrix} 3 \\ \Downarrow \end{matrix} \text{\hspace{0.34cm}} \\ \young(02,23,34) \end{pmatrix} \mapsto 
		\left( \raisebox{-0.35cm}{\text{$4 \Rightarrow$} } \young(02,23,34) \right) \mapsto
		\begin{pmatrix} \begin{matrix} 4\\ \Downarrow \end{matrix} \text{\hspace{0.34cm}} \\ \young(22,33,44) \end{pmatrix} \mapsto 
		\left( \raisebox{-0.35 cm}{\text{$5 \Rightarrow$} } \young(22,33,44) \right) \mapsto
		\left( \raisebox{-0.35 cm}{\text{$5 \Rightarrow$} } \young(23,34,45) \right) \mapsto \young(33,44,55) \nonumber \\
		&\begin{pmatrix} \begin{matrix} 2 \\ \Downarrow \end{matrix} \text{\hspace{0.34cm}} \\ \young(33,44,55) \end{pmatrix} \mapsto 
		\begin{pmatrix} \begin{matrix} 0 \\ \Downarrow \end{matrix} \text{\hspace{0.34cm}} \\ \young(23,44,55) \end{pmatrix} \mapsto 
		\begin{pmatrix} \begin{matrix} 3 \\ \Downarrow \end{matrix} \text{\hspace{0.34cm}} \\ \young(02,44,55) \end{pmatrix} \mapsto 
		\begin{pmatrix} \begin{matrix} 2 \\ \Downarrow \end{matrix} \text{\hspace{0.34cm}} \\ \young(02,34,55) \end{pmatrix} \mapsto 
		\begin{pmatrix} \begin{matrix} 4\\ \Downarrow \end{matrix} \text{\hspace{0.34cm}} \\ \young(02,23,55) \end{pmatrix} \mapsto 
		\begin{pmatrix} \begin{matrix} 3 \\ \Downarrow \end{matrix} \text{\hspace{0.34cm}} \\ \young(02,23,45) \end{pmatrix} \mapsto 
		\young(02,23,34) \label{eq:ex_pf_irreducibility}
		\end{align}
	}
	%\fi
	Now we show the assertion. Let $\mathtt{Y}_{t}=(\Gamma_{c}^{(a)}(t),\Gamma_{c+1}^{(a)}(t))$ be the joint carrier process over $X^{\mathbf{p}}$. To show (i), fix a tableau $C\in B_{c}^{(a)}$. According to \eqref{eq:column_bump_pf}, $U_{c}^{(a)}$ reaches $V_{c}^{(a)}$ and $V_{c}^{(a)}$ reaches $C$. More precisely, if we let $x_1\cdots x_k$ be the concatenation $\mathrm{row}'(V^{(a)}_c)\cdot \mathrm{row}'(C)$, then 
	\begin{align}
		R_{2}(U_{c}^{(a)}, x_{1}x_{2}\cdots x_{k}) = C. 
	\end{align}
	Denote $C'= R_{2}(U_{c+1}^{(a)}, x_{1}x_{2}\cdots x_{k}) \in B_{c+1}^{(a)}$. Then by the independence in $X^{\mathbf{p}}$ and since we are assuming $\mathbf{p}$ has strictly positive coordinates, it follows that 
	\begin{align}
	\PP(\mathtt{Y}_{t} = (C,C')) \ge \PP(X^{\mathbf{p}}(1)=x_{1},\cdots, X^{\mathbf{p}}(k)=x_{k}) = p_{x_{1}} \cdots p_{x_{k}} >0.
	\end{align}
	This shows (i). 
	
	For (ii), by the strong Markov property, it suffices to show that $\PP(\mathtt{Y}_{t}=(U_{c}^{(a)}, U_{c+1}^{(a)} ))>0$ for some $t\ge 1$. To this end, we apply \eqref{eq:column_bump_pf} with $c$ replaced by $c+1$ and $C=U_{c+1}^{(a)}$. This implies  
	\begin{align}
		R_{2}(U_{c+1}^{(a)}, \mathrm{row}'(V^{(a)}_{c+1})\cdot \mathrm{row}'(U_{c+1}^{(a)})) = U_{c+1}^{(a)}.
	\end{align} 
	On the other hand, using Proposition \ref{prop:LICI}, it is easy to see that 
	\begin{align}\label{eq:column_bump_pf2}
	R_2(U_{c}^{(a)},\mathrm{row}'(V^{(a)}_{c+1})) = V^{(a)}_{c}, \qquad
	R_2(V^{(a)}_{c}, \mathrm{row}'(U_{c+1}^{(a)})) = U_{c}^{(a)}.
	\end{align}
	This yields 
	\begin{align}
	R_{2}(U_{c}^{(a)}, \mathrm{row}'(V^{(a)}_{c+1})\cdot \mathrm{row}'(U_{c+1}^{(a)})) = U_{c}^{(a)}.
	\end{align}  
	Thus using a similar argument, we get $\PP(\mathtt{Y}_{2(c+1)}= (U_{c}^{(a)}, U_{c+1}^{(a)}) )>0$ as desired. 
	
	Lastly, (iii) follows immediately from (ii). By definition, $\Omega_{0}\subseteq \Omega$ consists of all jointly reachable states from $(U_{c}^{(a)}, U_{c+1}^{(a)})$. On the other hand, if $\mathtt{Y}_{t}=(C,C')$ for some $t\ge 0$, then by (ii), $\mathtt{Y}_{s}=(U_{c}^{(a)},U_{c+1}^{(a)})$ almost surely for some $s\ge t$. This shows the assertion. 
\end{proof}

\vspace{0.2cm}

\section{Proof of limit theorems for the rows}
\label{section:proof of main results}

In this section, we prove limit theorems for the rows: (SLLN) Theorem  \ref{thm:SLLN_rows} (i), (LDP) Theorem \ref{thm:row_concentration}, and (FCLT and persistence) Theorem \ref{thm:FCLT_persistence}. 

\vspace{0.2cm}
\subsection{SLLN for the energy matrix and the rows}

In the previous section, we have seen that the carrier process $\Gamma_{c}^{(a)}=(\Gamma(x))_{x\ge 0}$ over $X^{\mathbf{p}}$ defines an irreducible Markov chain on the finite state space $B_{c}^{(a)}$. In addition, the infinite i.i.d. basic $\kappa$-color BBS configuration $X^{\mathbf{p}}$ also defines a Markov chain on $B_{1}^{(1)}$, which is irreducible since $X^{\mathbf{p}}(x)$ is i.i.d. over $x$ and $\mathbf{p}$ is positive on every element of $B_{1}^{(1)}$. Then it follows that the pair $\mathtt{X}_{t}=(\Gamma(t-1),X^{\mathbf{p}}(t))$ defines an irreducible Markov chain on finite state space $B_{c}^{(a)}\times B_{1}^{(1)}$ with unique stationary distribution $\pi_{c}^{(a)}\otimes \mathbf{p}$, $(C,x)\mapsto \pi_{c}^{(a)}(C) \mathbf{p}(x)$. Note that its transition kernel $P$ is given by 
\begin{equation}\label{eq:kernel_carrier_box_MC}
P( (C, x), (C',x') ) = \mathbf{1}(C'=R_{2}(C,x)) \,\mathbf{p}(x'),
\end{equation}
where $R_{2}(C,x)\in B_{c}^{(a)}$ denotes the second component of the image $R(C,x)$ of the combinatorial $R$.

One of the key observation in the present paper is that the $(c,a)$-entry $E_{c}^{(a)}(X^{n,\mathbf{p}})$ of the energy matrix $E(X^{n,\mathbf{p}})$, which is defined by \eqref{eq:def_row_transfer_energy}, can be viewed as a Markov additive functional with respect to the Markov chain $\mathtt{X}_{t}=(\Gamma(t-1),X^{\mathbf{p}}(t))$ and the local energy $H$ defined at \eqref{eq:def_H} as the functional. This chain has also been introduced in a recent work \cite{kuniba2018randomized} in the more general class of BBS and the stationary distribution corresponding to $\pi_{c}^{(a)}$ in Theorem \ref{thm:stationary_measure} is obtained. However, irreducibility of the Markov chain and uniqueness of the stationary distribution has not been addressed in the general case. These properties are crucial in order to apply various limit theorems and large deviations principle to the energy matrix. 

With Theorem \ref{thm:stationary_measure} at hand, we can apply standard limit theorems for additive functionals of Markov chains for the row transfer matrix energy $E_{c}^{(a)}(X^{n,\mathbf{p}})$. Firstly in this section, we prove the strong law of large number (SLLN) for the energy matrix $E(X^{n,\mathbf{p}})$.

\begin{theorem}[SLLN for the energy matrix]\label{thm:SLLN_energy}
	Consider the basic $\kappa$-color BBS initialized at $X^{n,\mathbf{p}}$. For each integers $c\ge 1$ and $1\le a \le \kappa$, define a constant $\eps_{c}^{(a)}=\eps_{c}^{(a)}(\kappa,\mathbf{p})$ by 
	\begin{align}\label{eq:formula_stationary_energy}
	\eps_{c}^{(a)} = \frac{1}{Z_{c}^{(a)}} \sum_{C\in B_{c}^{(a)}} \prod_{j=0}^{\kappa} p_{j}^{m_{j}(C)} \left(\sum_{C(a,1)<i\le \kappa}p_{i}\right),
	\end{align}
	where $Z_{c}^{(a)}$ and $m_{j}(C)$ are as defined in Theorem \ref{thm:stationary_measure}, and $C(a,1)$ denotes the bottom left entry of the semistandard tableau $C$. Then $\eps_{c}^{(a)}\in (0,1)$ and almost surely as $n\rightarrow \infty$,
	\begin{equation}\label{eq:thm1_SLLN}
	n^{-1} E_{c}^{(a)}(X^{n,\mathbf{p}}) \rightarrow \eps_{c}^{(a)}.
	\end{equation} 
\end{theorem}

\begin{proof}
	Fix integers $a,c\ge 1$ and $\kappa\ge a$. Let $\eps_{c}^{(a)}$ be as defined at (\ref{eq:formula_stationary_energy}) and $H$ be the local energy function defined at \eqref{eq:def_H}. Consider the Markov chain $\mathtt{X}_{t} = (\Gamma(t-1),X^{\mathbf{p}}(t))$ on the finite state space $\Omega = B_{c}^{(a)} \times B_{1}^{(1)}$, where $\Gamma_{c}^{(a)}=(\Gamma(t))_{t\ge 0}$ is the carrier process over $X^{\mathbf{p}}$. Let $P$ be the transition kernel of this chain given in (\ref{eq:kernel_carrier_box_MC}). This chain is irreducible with unique stationary distribution $\pi=\pi_{c}^{(a)}\otimes \mathbf{p}$ by Theorem \ref{thm:stationary_measure}. Also, according to \eqref{eq:def_row_transfer_energy}, we can write $ E_{c}^{(a)}(\mathtt{X}^{n,\mathbf{p}}) = \sum_{k=1}^{n}H(\mathtt{X}_{t})$. It is well-known that a finite-state irreducible Markov chain is ergodic. So by the Markov chain ergodic theorem (e.g., \cite[Theorem 7.2.1]{Durrett}), almost surely, 
	\begin{equation}\label{eq:SLLN_pf}
	\lim_{n\rightarrow \infty} n^{-1} E_{c}^{(a)}(X^{n,\mathbf{p}}) = \mathbb{E}_{\pi_{c}^{(a)}\otimes \mathbf{p}}[H(C,x)].
	\end{equation}
	Furthermore, by Theorem \ref{thm:stationary_measure} and Proposition \ref{prop:LICI}, we have 
	\begin{align}
	\mathbb{E}_{\pi_{c}^{(a)}\otimes \mathbf{p}}[H(C,x)] &= \frac{1}{Z_{c}^{(a)}}\sum_{C\in B_{c}^{(a)}} \sum_{0\le i \le \kappa}\mathbf{1}(C(a,1)<i) p_{i} \prod_{j=0}^{\kappa} p_{j}^{m_{j}(C)}=\eps_{c}^{(a)}.
	\end{align}
	Lastly, observe that $H(C,0)=0$ for all $C\in B_{c}^{(a)}$ and $H(U_{c}^{(a)}, \kappa ) = 1$ due to Proposition \ref{prop:LICI}. Since $\pi_{c}^{(a)}\otimes \mathbf{p}$ is positive at every element in $B_{c}^{(a)}\times B_{1}^{(1)}$, this shows $\eps_{c}^{(a)}\in (0,1)$.  
\end{proof}

The quantity $\eps_{c}^{(a)}$ in 
(\ref{eq:formula_stationary_energy})  can be written as a ratio of Schur polynomials (see \eqref{qf}):
	\begin{align}
	\eps_{c}^{(a)} = \frac{s_{(c^{a},1)}(p_{0},\cdots,p_{\kappa}) }{s_{(c^{a})}(p_{0},\cdots,p_{\kappa})},
	\end{align}
	where $(c^{a},1)$ denotes the partition $(c,\cdots, c, 1)$ with $c$ repeated $a$ times. (See Example 3.3 and equation (22) in \cite{kuniba2018randomized} for a more general result). 

Now we derive Theorem \ref{thm:SLLN_rows} (i).

\begin{proof}[\textbf{Proof of Theorem \ref{thm:SLLN_rows} (i)}.]
	According to Theorem \ref{thm:SLLN_energy} 
	and the relation \eqref{eq:def_YDs}, almost surely,
	\begin{align}\label{eq:rescaled_row_lim}
		\lim_{n\rightarrow \infty} n^{-1}\rho_{i}^{(a)}(X^{n,\mathbf{p}})=  \eps_{i}^{(a)} - \eps_{i-1}^{(a)},
	\end{align}
	where we take $\eps_{0}^{(a)}=0$. Then we have 
	\begin{align}
	\eps_{i}^{(a)} - \eps_{i-1}^{(a)} &= \frac{s_{(i^{a},1)}(p_{0},\cdots,p_{\kappa}) }{s_{(i^{a})}(p_{0},\cdots,p_{\kappa})} -\frac{s_{((i-1)^{a},1)}(p_{0},\cdots,p_{\kappa}) }{s_{((i-1)^{a})}(p_{0},\cdots,p_{\kappa})} \\
	&=  \frac{s_{((i-1)^{a-1})}(p_{0},\cdots,p_{\kappa})\cdot s_{(i^{a+1})}(p_{0},\cdots,p_{\kappa})}{s_{(i^{a})}(p_{0},\cdots,p_{\kappa})\cdot s_{((i-1)^{a})}(p_{0},\cdots,p_{\kappa})},
	\end{align}
	where the last equality can be easily verified by using the Pl\"{u}cker relations. 
\end{proof}

\begin{example}\label{ex:energy_computation_single_color}
	Suppose $\kappa=a=1$ and denote $q=p_{1}/p_{0}$. Then
	\begin{align}
	Z_{c}^{(1)} = p_{0}^{0}p_{1}^{c}+p_{0}^{1}p_{1}^{c-1}+\cdots+p_{0}^{c}p_{1}^{0} = \frac{p_{0}^{c}(1-q^{c+1})}{1-q} \mathbf{1}(q\ne 1) + \frac{c+1}{2^{c}} \mathbf{1}(q=1). 
	\end{align}  
	Hence we obtain 
	\begin{align}
	\eps_{c}^{(1)} &= \mathbb{P}\left( \text{$C$ contains at least one 0} \right) = p_{1}\left( 1- \frac{p_{1}^{c}}{Z_{c}^{(1)}} \right) \\
	&= p_{1} \frac{1-q^{c}}{1-q^{c+1}} \mathbf{1}(p_{1}\ne 1/2)+ \frac{c}{2(c+1)}\mathbf{1}(p_{1}=1/2) \label{eq:eps_1BBS}. 
	\end{align}
	Then \eqref{eq:rescaled_row_lim} and a simple algebra shows \eqref{eq:thm1_eq3}. $\hfill \blacktriangle$
\end{example}

\subsection{Large deviations principle for the energy matrix and the rows}
\label{subsection:LDP}

In order to establish the large deviations principle for the row lengths, in view of the expression 
\eqref{eq:rho_MC_functional}, we introduce the following setup of Markov additive functionals over the joint carrier processes. For each pair of integers $c\ge 1$ and $1\le a \le \kappa$, let $\Gamma_{c}^{(a)}=(\Gamma_{c}^{(a)}(x))_{x\ge 0}$ denote the $B_{c}^{(a)}$-carrier process over $X^{\mathbf{p}}$. Define a Markov chain $(\mathtt{Z}_{t})_{t\ge 0}$ on the finite state space $\Sigma:=B_{c}^{(a)}\times B_{c+1}^{(a)} \times B_{1}^{(1)}$ by 
\begin{align}\label{eq:def_MC_Z}
\mathtt{Z}_{t} = (\Gamma_{c}^{(a)}(t),\Gamma_{c+1}^{(a)}(t), X^{\mathbf{p}}(t+1))
\end{align}
with initial state $\mathtt{Z}_{0}=(U_{c}^{(a)}, U_{c+1}^{(a)},X^{\mathbf{p}}(1))$. Its transition matrix $P:\Omega^{2}\rightarrow [0,1]$ is given by 
\begin{equation}\label{eq:kernel_carrier_box_MC_row}
P( (C_{1}, C_{2}, x), (C_{1}', C_{2}' ,x') ) = \mathbf{1}(C_{1}'=R_{2}(C_{1},x)) \mathbf{1}(C_{2}'=R_{2}(C_{2},x)) \,\mathbf{p}(x').
\end{equation}
Let $\Omega_{0}\subseteq B_{c}^{(a)}\times B_{c+1}^{(a)}$ denote the set of all jointly reachable states as in Lemma \ref{lemma:carrier_irreducibility}. Write $\Sigma_{0}:=\Omega_{0}\times B_{1}^{(1)}$. Then according to Lemma \ref{lemma:carrier_irreducibility}, $(\mathtt{Z}_{t})_{t\ge 0}$ defines an irreducible Markov chain on $\Sigma_{0}$. 

Fix an arbitrary functional $g:\Sigma_{0}\rightarrow \mathbb{R}$. Denote $\mathtt{S}_{n}=\sum_{k=0}^{n-1} g(\mathtt{Z}_{k})$, which defines a Markov additive functional. Since the underlying chain $(\mathtt{Z}_{t})_{t\ge 0}$ on state space $\Sigma_{0}$ is irreducible, this process satisfies the large deviations principle in the sense of Theorem \ref{thm:row_concentration} (i). Namely, for each $t\in \mathbb{R}$, let $P_{tg}$ be the exponentially weighted transition matrix defined by 
\begin{equation}\label{eq:def_P_tg}
P_{tg}(x,y)=P(x,y)e^{tg(y)}.
\end{equation}
Let $\lambda_{P}(tg)$ denote the largest real root of the characteristic polynomial $\phi(x,t)=\det(xI - P_{tg})$. By Perron-Frobenius theorem, $\lambda_{P}(tg)$ is positive and is a simple root of $\phi(x,t)$ for each $t\in \mathbb{R}$. Let $\Lambda(t):=\log \lambda_{P}(tg)$ and define its Legendre transform 
\begin{equation}\label{eg:def_legendre}
\Lambda^{*}(u) = \sup_{t\in \mathbb{R}} [ut - \Lambda(t)].
\end{equation}
It is well-known that both $\Lambda$ and $\Lambda^{*}$ are convex. Then Cram\'er's theorem for the Markov additive functionals asserts that (see  \cite[Theorem 3.1.2]{dembo2009large}) for any Borel set $F\subseteq \mathbb{R}$,  
\begin{equation}\label{eq:cramer}
-\inf_{u\in \mathring{F}} \Lambda^{*}(u) \le \liminf_{n\rightarrow \infty} \frac{1}{n} \log \mathbb{P}\left(n^{-1}\mathtt{S}_{n} \in F  \right)  \le \limsup_{n\rightarrow \infty} \frac{1}{n} \log \mathbb{P}\left(n^{-1}\mathtt{S}_{n} \in F  \right) \le -\inf_{u\in \bar{F}} \Lambda^{*}(u),
\end{equation}
where $\mathring{F}$ and $\bar{F}$ denotes the interior and closure of $F$, respectively. 

Now it is easy to derive Theorem \ref{thm:row_concentration} (i).
\begin{proof}[\textbf{Proof of Theorem \ref{thm:row_concentration} (i)}]
Define functionals $g_{E},g_{\rho}:\Sigma_{0}\rightarrow \mathbb{R}$ by 
\begin{align}
g_{E}(C_{1},C_{2}, x) &= H(C_{1},x), \label{eq:def_MC_g_E}\\
g_{\rho}(C_{1},C_{2}, x) &= H(C_{2},x) - H(C_{1},x). \label{eq:def_MC_g_rho}
\end{align}
Then according to \eqref{eq:def_row_transfer_energy} and \eqref{eq:rho_MC_functional}, we have 
\begin{align}\label{eq:energy_rho_MAP}
E_{c}^{(a)}(X^{n,\mathbf{p}}) = \sum_{k=0}^{n-1} g_{E}(\mathtt{Z}_{k}) , \qquad \rho_{c+1}^{(a)}(X^{n,\mathbf{p}}) = \sum_{k=0}^{n-1} g_{\rho}(\mathtt{Z}_{k}).
\end{align}
Hence both $E_{c}^{(a)}(X^{n,\mathbf{p}})$ and $\rho_{c+1}^{(a)}(X^{n,\mathbf{p}})$ satisfy the large deviations principle \eqref{eq:cramer} with rate functions $\Lambda^{*}$ defined at \eqref{eg:def_legendre} with the corresponding functionals $g_{E}$ and $g_{\rho}$, respectively. Moreover, since  $\rho_{1}^{(a)}(X^{n,\mathbf{p}})=E_{1}^{(a)}(X^{n,\mathbf{p}})$, this also covers the first row $\rho_{1}^{(a)}(X^{n,\mathbf{p}})$. This shows Theorem \ref{thm:row_concentration} (i).
\end{proof}

\begin{remark}
	For the large deviations rate function for the energy matrix $E(X^{n,\mathbf{p}})$, we can use the simpler setting of a single carrier process augmented with the following ball color $\mathtt{X}_{t}=(\Gamma(t-1),X^{\mathbf{p}}(t))$ defined at \eqref{eq:kernel_carrier_box_MC}, instead of the chain $\mathtt{Z}_{t}$ defined at \eqref{eq:def_MC_Z}. Indeed, according to \eqref{eq:def_row_transfer_energy}, we can write $E_{c}^{(a)}(X^{n,\mathbf{p}}) = \sum_{t=0}^{n-1} H(\mathtt{X}_{t})$, where $H:B_{c}^{(a)}\times B_{1}^{(1)}\rightarrow\{0,1\}$ be the local energy defined in (\ref{eq:def_H}). As the chain $\mathtt{X}_{t}$ is irreducible, $E_{c}^{(a)}(X^{n,\mathbf{p}})$ satisfies the large deviations principle as explained in \eqref{eq:cramer} and above, with $g=H$.
\end{remark}

\begin{example}\label{ex:3-color_row1}
	Let $\kappa=2$ and consider the $B_{1}^{(1)}$-carrier process $\Gamma_{1}^{(1)}$ with ball density $\mathbf{p}=(p_{0},p_{1},p_{2})$ (see Example \ref{example1}). The Markov chain $\mathtt{X}_{t}=(\Gamma(t-1),X^{\mathbf{p}}(t))$ is defined on the state space $B_{1}^{(1)}(2)\times B_{1}^{(1)}(2)$ with a $(9\times 9)$ transition kernel $P$ given in \eqref{eq:kernel_carrier_box_MC}. Let $g=H:B_{1}^{(1)}(2)\times B_{1}^{(1)}(2)\rightarrow\{0,1\}$ be the local energy defined in (\ref{eq:def_H}). We apply the large deviations principle for the Markov additive functional $E_{1}^{(1)}(X^{n,\mathbf{p}}) = \sum_{t=0}^{n-1}g(\mathtt{X}_{t})$. The corresponding exponentially weighted transition kernel $P_{tg}$ defined in \eqref{eq:def_P_tg} reads 
	\begin{align}
	P_{tg} = 
	\begin{matrix}
	(0,0)\\
	(1,0)\\
	(2,0)\\
	(0,1)\\
	(1,1)\\
	(2,1)\\
	(0,2)\\
	(1,2)\\
	(2,2)
	\end{matrix}
	\begin{bmatrix}
	p_{0} & 0 & 0 & p_{1} & 0 & 0 & p_{2} & 0 & 0 \\
	p_{0} & 0 & 0 & p_{1} & 0 & 0 & p_{2} & 0 & 0 \\
	p_{0} & 0 & 0 & p_{1} & 0 & 0 & p_{2} & 0 & 0 \\
	0 & p_{0}e^{t} & 0 & 0 & p_{1}e^{t} & 0 & 0 & p_{2}e^{t} & 0 \\
	0 & p_{0} & 0 & 0 & p_{1} & 0 & 0 & p_{2} & 0 \\
	0 & p_{0} & 0 & 0 & p_{1} & 0 & 0 & p_{2} & 0 \\
	0 & 0 & p_{0}e^{t} & 0 & 0 & p_{1}e^{t} & 0 & 0 & p_{2}e^{t} \\
	0 & 0 & p_{0}e^{t} & 0 & 0 & p_{1}e^{t} & 0 & 0 & p_{2}e^{t} \\
	0 & 0 & p_{0} & 0 & 0 & p_{1} & 0 & 0 & p_{2}	
	\end{bmatrix}
	.
	\end{align}
	Let $\Lambda(t)=\log\lambda_{P}(tg)$, where $\lambda_{P}(tg)$ is the largest real root of the following characteristic polynomial of $P_{tg}$ divided by $x^{6}$:
	\begin{align}
	x^{3} - x^{2}- x(e^{t}-1)(p_{0}p_{1}+p_{1}p_{2}+p_{2}p_{0}) -p_{0}p_{1}p_{2}(e^{t}-1)^{2}=0.
	\end{align}
	By taking partial derivative in $t$ and plugging in $(t,x)=(0,1)$, we get  
	\begin{equation}
	\eps_{1}^{(1)} = \frac{d\Lambda(t)}{dt}\bigg|_{t=0} =\frac{d\lambda_{P}(tH)}{dt} \bigg|_{t=0} = p_{0}p_{1}+p_{1}p_{2}+p_{2}p_{0}.  
	\end{equation} 	
	By Theorem \ref{thm:SLLN_rows}, this shows, almost surely, 
	\begin{align}
	\lim_{n\rightarrow \infty} n^{-1}\rho_{1}^{(1)}(n) = p_{0}p_{1}+p_{1}p_{2}+p_{2}p_{0}.
	\end{align}
	One can also compute this quantity in the more general case using the Schur function representation given in \eqref{eq:thm1_eq1}.
	
	Note that as $t\rightarrow -\infty$, $\lambda_{P}(tH)$ is asymptotically the largest real root of 
	\begin{align}
	x^{3} - x^{2}+x(p_{0}p_{1}+p_{1}p_{2}+p_{2}p_{0})-p_{0}p_{1}p_{2}=(x-p_{0})(x-p_{1})(x-p_{2})=0.
	\end{align}
	Thus,  $\Lambda(t)= \log \max(p_{0},p_{1},p_{2})+o(1)$ as $t\rightarrow -\infty$. On the other hand, as $t\rightarrow \infty$, $\lambda_{P}(tH)$ is asymptotically the largest real root of 
	\begin{align}
	x^{3}- xe^{t}(p_{0}p_{1}+p_{1}p_{2}+p_{2}p_{0})-p_{0}p_{1}p_{2} e^{2t}=0.
	\end{align}
	From this we deduce  
	\begin{equation}
	\Lambda(t) = \frac{2t}{3} + O(1)  \quad \text{as $t\rightarrow \infty$}.
	\end{equation}
	By definition (\ref{eg:def_legendre}), this implies that $\Lambda^{*}$ is strictly decreasing on $[0,p_{0}p_{1}+p_{1}p_{2}+p_{2}p_{0})$, strictly increasing on $(p_{0}p_{1}+p_{1}p_{2}+p_{2}p_{0},2/3]$, and equals $\infty$ elsewhere. $\hfill\blacktriangle$
\end{example}

\begin{example}\label{example2}
	Let $\kappa=2$ and consider the $B_{2}^{(1)}$-carrier process $\Gamma_{2}^{(1)}$ with uniform density $\mathbf{p}=(1/3,1/3,1/3)$. We borrow the setting in Example \ref{ex:3-color_row1}. The Markov chain $\mathtt{X}_{t}=(\Gamma_{2}^{(1)}(t-1),X^{\mathbf{p}}(t))$ is defined on the state space $B_{2}^{(1)}(2)\times B_{1}^{(1)}(2)$ with a $(18\times 18)$ transition kernel $P$ given in (\ref{eq:kernel_carrier_box_MC}). Let $\Lambda(t)=\log\lambda_{P}(tg)$, where $\lambda_{P}(tg)$ is the largest real root of the following characteristic polynomial of $P_{tg}$ divided by $x^{12}$:
	\begin{equation}\label{ex:char_poly}
	x^{6}-x^{5}-\frac{2e^{t}-1}{3} x^{4} - \frac{4e^{2t}-12e^{t}+1}{27} x^{3} + \frac{e^{t}(5e^{t}-2)}{27}x^{2} + \frac{2e^{2t}(e^{t}-2)}{81}x -\frac{e^{3t}(e^{t}+8)}{3^{6}}=0. 
	\end{equation}
	By taking partial derivative in $t$ and plugging in $(t,x)=(0,1)$, we get  
	\begin{equation}
	\eps_{2}^{(1)} = \frac{d\Lambda(t)}{dt}\bigg|_{t=0} =\frac{d\lambda_{P}(tg)}{dt} \bigg|_{t=0} = \frac{4}{9}.  
	\end{equation} 	
	
	Note that as $t\rightarrow -\infty$, $\lambda_{P}(tg)$ is asymptotically the largest real root of 
	\begin{equation}
	x^{3}(x-1/3)^{3} =0,
	\end{equation}
	which is $1/3$. Thus $\Lambda(t)= -\log 3+o(1)$ as $t\rightarrow -\infty$. On the other hand, as $t\rightarrow \infty$, $\lambda_{P}(tg)$ is asymptotically the largest real root of 
	\begin{equation}
	x^{6}- \frac{4e^{2t}-12e^{t}+1}{27} x^{3} -\frac{e^{3t}(e^{t}+8)}{3^{6}}=0, 
	\end{equation}
	so we get 
	\begin{equation}
	\Lambda(t) = \frac{2t}{3} - \log\left(\frac{3(\sqrt{5}-1)}{2}\right) +o(1) \quad \text{as $t\rightarrow \infty$}.
	\end{equation}
	By definition (\ref{eg:def_legendre}), this implies that $\Lambda^{*}$ is strictly decreasing on $[0,4/9)$, strictly increasing on $(4/9,2/3]$, and equals $\infty$ elsewhere. $\hfill\blacktriangle$
\end{example}

\begin{remark}
	Theorem \ref{thm:row_concentration} for $\kappa=1$ in fact strengthens \cite[Lemma 3.4]{levine2017phase}, where a polynomial concentration of $\rho_{c}^{(1)}(X^{n,\mathbf{p}})$ around its mean is shown for all $c\ge 1$.  Also, we remark that one can compute $\eps_{c}^{(1)}$ (and hence $\eta^{(1)}_c$) via a completely different method. Namely, let $(S_{k})_{k\ge 0}$ be the simple random walk started at $S_{0}=0$ and jumping to the right with probability $p$ and to the left with probability $1-p$. Let $\varsigma=\inf \{ k>0\,:\, S_{k}=0 \}$ be the first return time of $S_{k}$ to $0$. Then according to  \cite[Theorem 1]{levine2017phase}, we have 
	\begin{equation}
	\eps_{c}^{(1)} = \mathbb{P}\left(\max_{0\le k \le \varsigma} S_{k} \le  c\right). 
	\end{equation} 
	Now considering $S_{k}$ as the fortune of a gambler after betting $k$ times, using the Gambler's ruin probability \cite[Ch.\ 5.7]{Durrett}, we arrive at the same expression in (\ref{eq:eps_1BBS}).
\end{remark}

In the rest of this subsection, we show Theorem \ref{thm:row_concentration} (ii). The following proposition will be useful in showing that the rate function $\Lambda^{*}$ for $\rho_{i}^{(a)}(X^{n,\mathbf{p}})$ in our case is finite on an open interval containing 
$\eta^{(a)}_i$ defined in (\ref{eq:thm1_eq1}).

\begin{proposition}\label{prop:basic_est_thm1}

	\begin{description}
		\item[(i)] For any $\eps\in [0,\eta_{c}^{(a)})$ and $n\ge 1$, 
		\begin{equation}
		\mathbb{P}( n^{-1} \rho_{c}^{(a)}(X^{n,\mathbf{p}})\le \eps )\ge p_{0}^{n}.
		\end{equation}
				\item[(ii)] There exists constants $\nu \in (\eta_{c}^{(a)},1]$ and $\delta\in (0,1)$ such that for each $\eps \in (\eta_{c}^{(a)},\nu) $ and $n\ge 1$,
		\begin{equation}
		\mathbb{P}( n^{-1} \rho_{c}^{(a)}(X^{n,\mathbf{p}})\ge \eps )\ge \delta^{n}.
		\end{equation}
	\end{description}
\end{proposition}

\begin{proof}
	For (i), fix $0\le \eps<\eta^{(a)}_c$. Since $\Gamma_{c}^{(a)}(0)=U_{c}^{(a)}$ and $R(U_{c}^{(a)},0)=(0,U_{c}^{(a)})$ for each $c\ge 1$, independence in $X^{\mathbf{p}}$ gives 
	\begin{align}
	\mathbb{P}( n^{-1} \rho_{c}^{(a)}(X^{n,\mathbf{p}})\le \eps ) 
	& \ge \mathbb{P} (\rho_{c}^{(a)}(X^{n,\mathbf{p}}) = 0) \\
	& \ge  \mathbb{P}( X^{\mathbf{p}}(x)\equiv 0 \,\, \forall 1\le x \le n) \ge   p_{0}^{n}.
	\end{align}
	
	Next, we show (ii). Recall that by Lemma \ref{lemma:carrier_irreducibility}, the joint carrier process $(\Gamma_{c}^{(a)}(t), \Gamma_{c+1}^{(a)}(t))_{t\ge 0}$ on the set $\Omega_{0}$ of all jointly reachable states is irreducible. Since $\Omega_{0}$ is finite, it has a unique stationary distribution, which we will denote by $\pi_{c,c+1}^{(a)}$. It follows that the Markov chain $\mathtt{Z}_{t}$ (see \eqref{eq:def_MC_Z}) on state space $\Sigma_{0}=\Omega_{0}\times B_{1}^{(1)}$ is irreducible with unique stationary distribution $\pi:=\pi_{c,c+1}^{(a)}\otimes \mathbf{p}$. 
	
	Fix the integers $c,L\ge 1$. For each $x\in \mathbb{N}$, we call the interval $[x,x+L)$ `frozen' if $\Gamma_{c}^{(a)}(x)=U_{c}^{(a)}$, $\Gamma_{c+1}^{(a)}(x)=U_{c+1}^{(a)}$, and $X^{\mathbf{p}}(y)\equiv 0$ for all $y\in [x,x+L)$. Let $g_{E}$ and $g_{\rho}$ denote the functionals defined at \eqref{eq:def_MC_g_E} and $\eqref{eq:def_MC_g_rho}$, and recall the relations \eqref{eq:energy_rho_MAP}. According to Proposition \ref{prop:LICI}, it follows that, if $[x,x+L)$ is frozen, then
	\begin{align}
	g_{E}(\mathtt{Z}_{y}) = g_{\rho}(\mathtt{Z}_{y}) = 0 \qquad \forall y\in [x,x+L).
	\end{align}
	Observe that if $L$ is large enough, we can modify the configuration $X^{\mathbf{p}}$ over the frozen interval $[x, x+L)$ in such a way that we gain additional unit contribution both to the energy $E_{c}^{(a)}(X^{n,\mathbf{p}})$ and the row length $\rho_{c+1}^{(a)}(X^{n,\mathbf{p}})$, while maintaining the chain $\mathtt{Z}_{t}$ the same at the beginning and at the end of such interval. 
	
	Suppose $L$ is large enough so that the above mentioned replacement holds for each frozen interval $[x,x+L)$. Denote 
	\begin{align}
		N_{L}(n) = \sum_{1\le k< n/L} \mathbf{1}\big( \text{$[kL, (k+1)L)$ is frozen} \big),
	\end{align}
	which is the number of frozen intervals of the form $[kL,(k+1)L)$ contained in the interval $[1,n]$. Since the chain $(\mathbb{Z}_{t})_{t\ge 0}$ is irreducible, the $L$-fold Markov chain $(\mathtt{Z}_{kL},\cdots, \mathtt{Z}_{(k+1)L-1})_{k\ge 1}$ is also irreducible and hence ergodic. So by the Markov chain ergodic theorem,  
	\begin{align}
	\lim_{n\rightarrow \infty}\frac{1}{n}N_{L}(n) > 0 
	\end{align}
	almost surely. Denoting the limit in the left hand side by $\tau\in (0,1)$, it follows that there exists a constant $K\in \mathbb{N}$ such that 
	\begin{align}\label{eq:density_frozen_intervals}
	\PP\left( N_{L}(n) \ge (\tau/2) n \quad \text{for all $n\ge K$}  \right) = 1.
	\end{align}
	
	Now for each $n\ge K$, there are at least $(\tau/2)n$ frozen intervals of the form $[kL, (k+1)L)$ contained in $[1,n]$ with probability 1. By the independence in $X^{\mathbf{p}}$, we can perform such replacement independently over all such frozen intervals with a positive probability, say, $\delta$. Hence we may perform the replacement for the $(\tau/2) n$ frozen intervals appearing in $[1,n]$ to increase both $E_{c}^{(a)}(X^{n,\mathbf{p}})$ and  $\rho_{c+1}^{(a)}(X^{n,\mathbf{p}})$ by $(\tau/2) n$ at the cost of exponential probability $\delta^{n}>0$. As $\rho_{1}^{(a)}(X^{n,\mathbf{p}}) = E_{1}^{(a)}(X^{n,\mathbf{p}})$ and $c\ge 1$ was arbitrary, it follows that 
	\begin{equation}\label{eq:pf_basic_est}
	\mathbb{P}\big( n^{-1} \rho_{c}^{(a)}(X^{n,\mathbf{p}})\ge \eta_{c}^{(a)}+(\tau/2) \big)\ge \delta^{n}
	\end{equation}
	for all $c\ge 1$ and $n\ge K$. Note that since $\rho_{c}^{(a)}(X^{n,\mathbf{p}})\le n$, this also implies $\eta_{c}^{(a)}+(\tau/2)\le 1$. Hence (ii) follows by letting $\nu=\eta_{c}^{(a)}+(\tau/2)$.   
\end{proof}

Now we are ready to prove Theorem \ref{thm:row_concentration} (ii).

\begin{proof}[\textbf{Proof of Theorem \ref{thm:row_concentration} (ii)}]
	We show the assertion for $c\ge 2$, and a similar argument applies for the case $c=1$. Let $(\mathtt{Z}_{t})_{t\ge 0}$ be the irreducible Markov chain on state space $\Sigma_{0}=\Omega_{0}\times B_{1}^{(1)}$ defined at \eqref{eq:def_MC_Z}. Let $\pi=\pi_{c,c+1}^{(a)}\otimes \mathbf{p}$ denote its unique stationary distribution. Let $g=g_{\rho}:\Sigma\rightarrow \mathbb{R}$ be the functional defined at \eqref{eq:def_MC_g_rho}. Let $\lambda_{P}(tg)$, $\Lambda$, and $\Lambda^{*}$ be as defined at \eqref{eg:def_legendre} and above.

	Recall that $\lambda_{P}(tg)$ is a positive simple root of the characteristic polynomial $\phi(x,t)=\det(xI - P_{tg})$ for each $t\in \mathbb{R}$. Hence by the Implicit Function Theorem for analytic functions (see, e.g., \cite[Thm. 2.1.2]{hormander1973introduction} or \cite[Thm 2.3.5]{krantz2002primer}), $\lambda_{P}(tg)$ is an analytic function in $t$ on $\mathbb{R}$. Also recall that both $\Lambda$ and $\Lambda^{*}$ are convex. Furthermore, it holds that 
	\begin{align}
	\frac{d}{dt} \Lambda(t) \bigg|_{t=0} = \EE_{\pi}[g(\mathtt{Z}_{0})] = \lim_{n\rightarrow \infty} n^{-1}\rho_{c}^{(a)}(X^{n,\mathbf{p}}) = \eta_{c}^{(a)},
	\end{align} 
	where we have used the Markov chain ergodic theorem and Theorem \ref{thm:SLLN_rows} for the last two equalities.  This yields $\Lambda^{*}(\eta_{c}^{(a)}) = -\Lambda(0)=  0$. Moreover, $\Lambda^{*}(u)\ge -\Lambda(0)=0$ for all $u\in \mathbb{R}$. Since $\Lambda^{*}$ is convex, it follows that $\Lambda^{*}$ is non-increasing over $(-\infty,\eta_{c}^{(q)}]$ and non-decreasing over $[\eta_{c}^{(a)},\infty)$.
	
	Let $\nu\in (\eta_{c}^{(a)},1]$ be the constant in Proposition \ref{prop:basic_est_thm1}. It suffices to show that $\Lambda^{*}\in (0,\infty)$ on $[0,\eta_{c}^{(a)})\cup (\eta_{c}^{(a)},\nu]$. First, we show that $\Lambda^{*}<\infty$ over $[0,\nu]$. Fix $0\le \eps<\eta_{c}^{(a)}$. According to Proposition \ref{prop:basic_est_thm1} (i) and \eqref{eq:cramer_thm} we have shown above, it follows that 
	\begin{align}
	\log p_{0} \le \limsup_{n\rightarrow \infty} \frac{1}{n} \log \PP\left( n^{-1}\rho_{c}^{(a)}(X^{n,\mathbf{p}})\le \eps  \right) = -\inf_{u\le \eps} \Lambda^{*}(u) = -\Lambda^{*}(\eps).
	\end{align}
	Note that the last equality follows since $\Lambda^{*}$ is non-increasing on $(-\infty, \eta_{c}^{(a)})$. Hence $\Lambda^{*}(\eps)\le \log p_{0}^{-1}<\infty$ for all $\eps\in [0,\eta_{c}^{a})$. Similarly, for any $\eps\in (\eta_{c}^{(a)},\nu]$, as $\Lambda^{*}$ is non-decreasing on $(\eta_{c}^{(a)},\infty)$, Proposition \ref{prop:basic_est_thm1} (ii) yields 
	\begin{align}
	\log \delta \le \limsup_{n\rightarrow \infty} \frac{1}{n} \log \PP\left( n^{-1}\rho_{c}^{(a)}(X^{n,\mathbf{p}})\ge\eps  \right) = -\inf_{u\ge \eps} \Lambda^{*}(u) = -\Lambda^{*}(\eps)
	\end{align}
	for some constant $\delta>0$. So $\Lambda^{*}(\eps)\le \log \delta^{-1}<\infty$ for any $\eps\in [0,\nu]$.
	
	Finally, we show $\Lambda^{*}>0$ over $[0,\eta_{c}^{(a)})\cup (\eta_{c}^{(a)},\nu]$. For contradiction, suppose $\Lambda^{*}(\eps)=0$ for some $\eps\in  [0,\eta_{c}^{(a)})\cup (\eta_{c}^{(a)},\nu]$. First suppose $\eps\in (\eta_{c}^{(a)},\nu]$. Since $\Lambda(t)$ is differentiable, this yields 
	\begin{align}\label{eq:of_rate_positive_strictly_convex}
	\Lambda^{*}(\eps) = \eps t_{\eps} - \Lambda(t_{\eps}),
	\end{align} 
	where $t_{\eps} \in \mathbb{R}$ satisfies $\Lambda'(t_{\eps})=\eps$. Since $\Lambda'(t_{\eps}) = \eps>\eta^{(a)}_{c} = \Lambda'(0)$ and $\Lambda$ is convex, it follows that $t_{\eps}> 0$. Moreover, by the convexity of $\Lambda$, for any $s\in (0,t_{\eps})$, 
	\begin{align}
	\eps = \frac{\Lambda(t_{\eps})}{t_{\eps}} \le \frac{\Lambda(t_{\eps}) - \Lambda(s)}{t_{\eps}-s} \le \Lambda'(t_{\eps}) = \eps,
	\end{align}
	so \eqref{eq:of_rate_positive_strictly_convex} with the assumption $\Lambda^{*}(\eps)=0$ yield 
	\begin{align}
	\Lambda(s) = \eps s + \Lambda(t_{\eps}) - \eps t_{\eps} = \eps s - \Lambda^{*}(\eps) = \eps s.
	\end{align}
	Thus $\lambda_{P}(tg) = e^{\eps t}$ for all $t\in (0,t_{\eps})$. But since $\lambda_{P}(tg)$ is analytic on $\mathbb{R}$ and $t_{\eps}>0$, by the identity theorem (see, e.g., \cite[Cor. 1.2.6]{krantz2002primer}), it follows that $\lambda_{P}(tg) = e^{\eps t}$ for all $t\in \mathbb{R}$. This yields 
	\begin{align}
	\eta_{c}^{(a)} = \Lambda'(0) = \frac{d}{dt} \log e^{\eps t} \bigg|_{t=0} = \eps,	
	\end{align}
	which is a contradiction. A similar argument rules out the other case $\eps\in [0,\eta_{c}^{(a)})$. This shows the assertion.
\end{proof}

\subsection{Fluctuation of the rows}

Having established the strong law of large numbers and large deviations principle for the row lengths, a natural next question is about their fluctuation around their mean. In Theorem \ref{thm:FCLT_persistence}, we characterize the fluctuation of the rows in terms of the functional central limit theorem (FCLT) and their `persistence scaling'. Namely, the central limit theorem about the rows  states the following convergence in distribution 
\begin{equation}
\frac{\rho_{c}^{(a)}(X^{n,\mathbf{p}}) - \eta^{(a)}_cn }{ \sqrt{n} } \Longrightarrow \gamma_{g}Z,
\end{equation}   
where $Z$ is the standard normal random variable and $\gamma_{g}$ is a quantity that will be defined shortly. The FCLT gives a stronger version of such diffusive scaling in the process level. On the other hand, persistence scaling gives the asymptotic behavior as $n\rightarrow \infty$ of the probability that the energy $\rho_{c}^{(a)}(X^{k,\mathbf{p}})$ 
beats its mean $\eta^{(a)}_c k$ for all $1\le k \le n$. 

To make precise statements, fix $c\ge 1$ and $1\le a \le \kappa$, and let $(\mathtt{Z}_{t})_{t\ge 0}$ be the irreducible Markov chain on state space $\Sigma_{0}=\Omega_{0}\times B_{1}^{(1)}$ defined at \eqref{eq:def_MC_Z}. Denote by  $\pi=\pi_{c,c+1}^{(a)}\otimes \mathbf{p}$ its unique stationary distribution. Let $g=g_{\rho}:\Sigma\rightarrow \mathbb{R}$ be the functional defined at \eqref{eq:def_MC_g_rho}. By the independence in $X^{\mathbf{p}}$, it is not hard to see that the following quantity  
\begin{equation}\label{def:limiting_variance}
\gamma_{g}^{2} := \text{Var}_{\pi}[g(\mathtt{Z}_{0})]+2\sum_{k=1}^{\infty}\text{Cov}_{\pi}[g(\mathtt{Z}_{0}),g(\mathtt{Z}_{k})]
\end{equation}
is positive and finite. Here $\gamma^{2}_{g}$ is called the $\textit{limiting variance}$ of the additive process $\sum_{k=0}^{n-1}g(\mathtt{Z}_{k})=\rho_{c}^{(a)}(X^{n,\mathbf{p}})$. Denote by $\gamma_{g}$ the positive square root of $\gamma_{g}^{2}$. 

\begin{theorem}\label{thm:FCLT_persistence}
	Fix integers $c\ge 1$ and $1\le a \le \kappa$. Denote $\bar{\rho}(n)=\rho_{c}^{(a)}(X^{n,\mathbf{p}}) - \eta^{(a)}_c n$ for each $n\ge 1$. Then the followings hold.
	\begin{description}  
		\item[(i)] \textup{(FCLT)} Let $[\bar{\rho}](\cdot):[0,\infty)\rightarrow \mathbb{R}$ denote the linear interpolation of the points $(n,\bar{\rho}(n))_{n\ge 0}$ with setting $\bar{\rho}(0)=0$. Let $C[0,1]$ be the set of all continuous functions $[0,1]\rightarrow \mathbb{R}$. Then for any continuous functional $F:C[0,1]\rightarrow \mathbb{R}$, 
		\begin{equation}
		\lim_{n\rightarrow \infty} \mathbb{E}_{\pi}[ F( n^{-1/2}[\bar{\rho}](ns):0\le s \le 1 ) ] = \mathbb{E}_{\pi}[ F(\gamma_{g} B)],
		\end{equation}
		where $B=(B_{s}\,:\, 0\le s\le 1)$ is the standard Brownian motion.
		\vspace{0.1cm}
		\item[(ii)] \textup{(Persistence)} As $n\rightarrow \infty$,
		\begin{equation}
		\mathbb{P}\left( \bar{\rho}(1)\ge 0,\cdots, \bar{\rho}(n)\ge 0 \,|\, \bar{\rho}(0)=0 \right) \sim \frac{\gamma_{g}}{(1-\eps_{c}^{(a)})\sqrt{2\pi}} n^{-1/2}.
		\end{equation}
	\end{description}
\end{theorem}

\begin{proof}
	Note that the Markov chain $\mathtt{Z}_{t}$ defined at \eqref{eq:def_MC_Z} has the initial state $(U_{c}^{(a)}, U_{c+1}^{(a)}, X^{\mathbf{p}}(1))$, which is not distributed according to its stationary distribution $\pi$. However, if we let $\mathtt{Z}'_{t}$ denote an independent copy of $\mathtt{Z}_{t}$ started with the initial distribution $\pi$, then we can couple the two chains by independently evolving them until the first time they meet, and jointly evolving them in synchrony thereafter. Since they are irreducible Markov chains on the finite state space $\Sigma_{0}$, such a meeting time is almost surely finite. It follows that it suffices to show the assertion for the stationary version $\mathtt{Z}_{t}'$. For (i), we refer to \cite[Corollary 3]{dedecker2000functional} or \cite[Theorem 17.4.4]{meyn2012markov}; For (ii), see Lyu and Sivakoff \cite[Theorem 2]{lyu2018persistence}.  
\end{proof}

\vspace{0.3cm}
\section{Conditioning on the highest states}
\label{section:highest_states}

In this section we prove Theorem  \ref{thm:SLLN_rows} (ii). The key question we need to address is the following: \textit{How probable it is to sample a highest state from the i.i.d. basic $\kappa$-color BBS configuration $X^{n,\mathbf{p}}$?} This question can be phrased in terms of lattice paths as follows. Let $\mathbf{e}_{1},\cdots, \mathbf{e}_{\kappa+1}$ be the standard basis vectors of the $(\kappa+1)$-dimensional lattice $\mathbb{Z}^{\kappa+1}$. The \textit{Weyl chamber} of dimension $\kappa+1$ is the subset $W_{\kappa+1}:=\{ (x_{1},\cdots, x_{\kappa+1})\in \mathbb{N}_{0}^{\kappa+1}\,:\, x_{1}\ge x_{2} \ge \cdots \ge x_{\kappa+1}  \}$. Given a basic $\kappa$-color BBS configuartion $X_{0}:\mathbb{N}\rightarrow \{0,1,\cdots,\kappa \}$, define an associated $(\kappa+1)$-dimensional lattice path $(S_{k})_{k\ge 0}$ by $S_{0}=(0,\cdots,0)$ and 
\begin{equation}
S_{k}-S_{k-1} = \sum_{i=0}^{\kappa} \mathbf{e}_{i+1} \mathbf{1}(X_{0}(k)=i).
\end{equation}
Namely, if there is a ball of color $i-1\ge 0$ (calling empty box as ball of color $0$) at $k^{\text{th}}$ box, then the walk moves a unit step in the $i^{\text{th}}$ direction. Then observe that $X_{0}$ is a highest state if and only if $S_{k}$ lies entirely in the Weyl chamber.

For example, if we let $\kappa=1$, $\mathbf{p}=(1-p,p)$, and $X_{0}=X^{n,\mathbf{p}}$,  then $S_{k}$ becomes a simple directed random walk in $\mathbb{N}_{0}^{2}$. If $p>1/2$, certainly the probability of sampling a path that is contained in the 2-dimeionsional Weyl chamber $W_{2}$ will decay to zero exponentially fast as $n\rightarrow \infty$. If $0<p<1/2$, then this probability converges to a nontrivial value in $(0,1)$. For the general case $\kappa\ge 1$, when the ball densities strictly decrease, an elementary probabilistic argument shows that $X^{n,\mathbf{p}}$ is a highest state for all $n\ge 1$ with a positive probability. This is the content of the following proposition.

\begin{proposition}\label{prop:prob.highest_strict}
	If $p_{0}>p_{1}>\cdots>p_{\kappa}$, then 
	\begin{equation}
	\mathbb{P}\left( S_{k}\in W_{\kappa+1} \,\, \forall k\ge 0 \right)>0.
	\end{equation}
\end{proposition}	

\begin{proof}
	For each integer $i,j\ge 1$, define a simple random walk $S_{k}^{ij}$ by $S_{0}^{ij}=0$ and 
	\begin{equation}
	S_{k}^{ij} = \#( \text{balls of color $i$ in $[1,k]$} ) - \#( \text{balls of color $j$ in $[1,k]$} ).
	\end{equation}
	If $i<j$, then the random walk $S_{k}^{ij}$ has positive drift so there exists a finite random time $\tau^{ij}$ such that $S_{k}-S_{\tau^{ij}}\ge 0$ for all $k\ge \tau^{ij}$. Define a random time $\tau=\max_{0\le j<i\le \kappa}( \tau^{ij})$. Then $\tau$ is almost surely finite, so we may choose $N_{1}\ge 1$ large enough so that 
	\begin{equation}\label{eq:highest_strict_1}
	\mathbb{P}(\tau \le N_{1}) >\delta_{1}>0.
	\end{equation} 
	By using independence of the ball colors, choose $N_{2}\ge 1$ large enough so that 
	\begin{equation}
	\mathbb{P}\left( \text{$S^{ij}_{N_{2}} > 2N_{1}$ for all $0\le i<j\le \kappa$} \right) > \delta_{2}>0. 
	\end{equation}
	In other words, $S_{N_{2}}\in W_{\kappa+1}$ and is at least $2N_{1}$ steps away from the boundary of $W_{\kappa+1}$. Hence $S_{N_{1}+N_{2}}\in W_{\kappa+1}$ with probability $>\delta_{2}$. Since the restarted lattice path $(S_{k})_{k> N_{2}}$ has the same law as $(S_{k})_{k\ge 0}$ by the Markov property, by using the independence of increments and  (\ref{eq:highest_strict_1}), we get  
	\begin{align}
	&\mathbb{P}\left( \text{$S_{k}\in W_{\kappa+1}$ for all $k\ge 0$ } \right) \\
	&\qquad \ge \mathbb{P}(S_{N_{1}+N_{2}}\in W_{\kappa+1}) \mathbb{P}\left( S_{k+N_{1}}-S_{N_{1}+N_{2}} \in W_{\kappa+1}\,\, \forall k\ge N_{2} \right) \ge \delta_{1}\delta_{2}>0.
	\end{align}
	This shows the assertion.
\end{proof}

\begin{remark}
	When $\kappa=1$ and $p_{0}>p_{1}$, a simple application of Wald's equation \cite[Exercise 4.1.13]{Durrett} gives 
	\begin{equation}
	\mathbb{P}(S_{k}\in W_{2} \,\, \forall k\ge 0) = \frac{2p_{0}-1}{p_{0}}.
	\end{equation}
\end{remark}

Now consider the critical case $p_{1}=1/2$ when $\kappa=1$. Then the resulting 2-dimensional lattice walk has no bias, and the probability that $S_{k}\in W_{2}$ for all $k\ge 1$ is the same as the `survival probability' of a simple symmetric random walk, and the following asymptotic is known (see, e.g.,  \cite[Theorem XII.7.1a]{feller1971introduction}):
\begin{equation}
\mathbb{P}(\text{$(S_{k})_{1\le k \le n}$ is contained in $W_{2}$}) \sim \sqrt{\frac{2}{\pi n}}.
\end{equation}
In general, if some of the inequalities $p_{0}\ge p_{1}\ge \cdots \ge p_{\kappa}$ are equalities, then we expect the probability that $S_{k}$ stays in $W_{\kappa+1}$ for $n$ steps decays to zero at a polynomial rate. This claim can be justified by using a high dimensional Ballot theorem, as we shall explain below. 

Suppose $p_{1}=1/2$ and after $n$ steps we knew that there were $n_{0}$ empty boxes and $n_{1}$ balls in the initial configuration. Then what is the probability of a highest state given that there are $n_{0}$ empty boxes and $n_{1}$ balls? This is a classic problem in probability theory known as the Ballot problem (see., e.g., \cite[Ch. III.1]{feller1957introduction}). This problem has been generalized to the higher dimensions and addressed in \cite{zeilberger1983andre}, giving a general Ballot theorem that enumerates the number of lattice paths within the Weyl chamber from the origin to a fixed destination. Namely, for each point $\mathbf{m}=(m_{1},\cdots,m_{r})\in W_{r}$ in the $r$-dimensional Weyl chamber, let $G(\mathbf{m})$ be the number of lattice paths from the origin to $\mathbf{m}$ that is contained in the Weyl chamber. Then
\begin{equation}\label{eq:ballot_ndim}
G(\mathbf{m}) = (m_{1}+\cdots+m_{r})!  \frac{\prod_{1\le i<j\le r}(m_{i}-m_{j}+j-i)}{(m_{1}+r-1)! (m_{2}+r-2)! \cdots m_{r}!}.
\end{equation}
Using this enumeration, we can estimate the probability of sampling a highest state from an i.i.d. BBS configuration.  

\begin{proposition}\label{prop:prob_highest}
	Suppose $p_{0}\ge p_{1}\ge \cdots \ge p_{\kappa}$. Then there exists constants $c_{1},c_{2}>0$ such that for all $n\ge 1$,
	\begin{equation}
	c_{1}n^{-a/2} \le \mathbb{P}( \text{$X^{n,\mathbf{p}}$ is a highest state}) \le c_{2} n^{-a/2},
	\end{equation}
	where $a=\sum_{0\le i < j \le \kappa} \mathbf{1}(p_{i}=p_{j})$.
\end{proposition}

\begin{proof}[Sketch of proof]
	For simplicity of the argument, we show the assertion for the special case of $\kappa=2$. Observe that each sample path with $m_{0}$ empty boxes, $m_{1}$ balls of color 1, and $m_{2}$ balls of color 2 occurs with the same probability $p_{0}^{m_{1}} p_{1}^{m_{2}} p_{2}^{m_{3}}$. Hence using \eqref{eq:ballot_ndim}, we can write 
	\begin{eqnarray}
	&&\mathbb{P}( \text{$X^{n,\mathbf{p}}$ is a highest state}) \\
	&&\qquad  = \sum_{\substack{m_{1}\ge m_{2} \ge m_{3}\ge 0\\ m_{1}+m_{2}+m_{3}=n}} \frac{(m_{1}-m_{2}+1)(m_{2}-m_{3}+1) (m_{1}-m_{3}+2)}{ (m_{1}+1)(m_{1}+2)(m_{2}+1)} p_{0}^{m_{1}} p_{1}^{m_{2}} p_{2}^{m_{3}}  .
	\end{eqnarray} 
	By law of large numbers, each ball count $m_{i}$ is of order $O(n)$. If $p_{0}>p_{1}>p_{2}$, then each $m_{i}-m_{j}$ for $0\le i<j\le 2$  grows asymptotically linearly in $n$, so the above summation is of order $O(1)$. On the other extreme, suppose $p_{0}=p_{1}=p_{2}=1/3$. Then the sum gets the most contribution when each $m_{i}-m_{j}$ is of order $O(\sqrt{n})$, so in this case the sum is of order $O(n^{-3/2})$. For the general case, each term $m_{i}-m_{j}=O(\sqrt{n})$ contributes a factor of $O(n^{-1/2})$, so the general asymptotic formula follows in this manner. 
\end{proof}

Now we show Theorem \ref{thm:SLLN_rows} (ii) from Theorem \ref{thm:row_concentration} and Proposition \ref{prop:prob_highest}.

\begin{proof}[\textbf{Proof of Theorem \ref{thm:SLLN_rows} (ii)}.]
	
	Let $\Lambda^{*}$ be the rate function in Theorem  \ref{thm:row_concentration}, which is positive in an open neighborhood of $\eta_{c}^{(a)}$ defined in (\ref{eq:thm1_eq1}).
	Write 
	\begin{equation}\label{cdk}
	\tilde{\Lambda}^{*}(u) = \min( \Lambda^{*}(\eta_{c}^{(a)}-u),  \Lambda^{*}( \eta_{c}^{(a)}+u)).
	\end{equation}
	Fix $0\le u \le \eta_{c}^{(a)}$ and let $I = [0,\eta_{c}^{(c)}-u)\cup (\eta_{c}^{(a)}+u,\infty)$. Then Theorem \ref{thm:row_concentration} yields that 
	\begin{align}
		\lim_{n\rightarrow \infty} n^{-1}\log  \mathbb{P}\left( \left| n^{-1}\rho_{c}^{(a)}(X^{n,\mathbf{p}}) - \eta_{c}^{(a)} \right|\ge u \right) = -\inf_{s \in I}\Lambda^{*}(s) = \tilde{\Lambda}^{*}(u),
	\end{align}
	where the last equality uses the fact that $\Lambda^{*}$ is non-increasing in $[0,\eta_{c}^{(a)})$ and non-decreasing in $[\eta_{c}^{(a)},\infty)$. Hence there exists $N(u)\ge 1$ such that 
	\begin{align}
	\mathbb{P}\left( \left| n^{-1}\rho_{c}^{(a)}(X^{n,\mathbf{p}}) - \eta_{c}^{(a)} \right|\ge u \right) &= \exp\left[  n\left( n^{-1}\log  \mathbb{P}\left( \left| n^{-1}\rho_{c}^{(a)}(X^{n,\mathbf{p}}) - \eta_{c}^{(a)} \right|\ge u \right) \right) \right]\\
	&\le \exp\left( {-(\tilde{\Lambda}^{*}(u)/2)n}\right)
	\end{align}
	for all $n\ge N(u)$.

	Suppose $p_{0} \ge p_{1}\ge \cdots\ge p_{\kappa}$. Then by Proposition \ref{prop:prob_highest}, there exists a constant $c_{1}>0$ such that  
	\begin{align}
	\mathbb{P}\left( \left| n^{-1}\rho_{c}^{(a)}(X^{n,\mathbf{p}}) - \eta_{c}^{(a)} \right|\ge u \,\bigg|\, \text{$X^{n,\mathbf{p}}$ is highest}  \right) &\le \frac{\mathbb{P}\left(  \left| n^{-1}\rho_{c}^{(a)}(X^{n,\mathbf{p}}) - \eta_{c}^{(a)} \right|\ge u \right) }{ \mathbb{P}( \text{$X^{n,\mathbf{p}}$ is highest} ) }\\
	&\le  c_{1} n^{\kappa(\kappa+1)/2}  \exp(- (\tilde{\Lambda}^{*}(u)/2) n), \label{eq:conditioned_concentration_eq}
	\end{align}
	for sufficiently small $u>0$ and all $n\ge N(u)$.

	By using the continuity of $\Lambda^{*}$ at $\eta_{c}^{(a)}$ and its positivity near $\eta_{c}^{(a)}$, we may choose a sequence $u_{r}>0$ such that $u_{r}\rightarrow 0$ and $\tilde{\Lambda}^{*}(u_{r})\ge r^{-1/2}$ as $r\rightarrow \infty$. Then we have 
	\begin{align}
	\mathbb{P}\left( \left| n^{-1}\rho_{c}^{(a)}(X^{n,\mathbf{p}}) - \eta_{c}^{(a)} \right|\ge u_{r} \,\bigg|\, \text{$X^{n,\mathbf{p}}$ is highest}  \right) &\le  c_{1} n^{\kappa(\kappa+1)/2} \exp(- n/\sqrt{r})
	\end{align}
	for all $r\ge 1$ and $n\ge N(u_{r})$. Note that the right hand side is summable for any sequence $n=n_{r}$ as long as $n_{r}\rightarrow \infty$ as $r\rightarrow \infty$. This implies that for any sequence $(n_{r})_{r\ge 1}$ of integers such that $n_{r}\ge N(u_{r})$, the above probabilities are summable over all $n_{r}$. By Borel-Cantelli lemma, this yields 
	\begin{equation}\label{eq:corollary_last}
	n^{-1}\rho_{c}^{(a)}(X^{n_{r},\mathbf{p}}) \rightarrow \eta_{c}^{(a)}
	\end{equation}
	almost surely as $r\rightarrow \infty$. It follows that if we are given any sequence $n_{k}\rightarrow \infty$, we may take a further subsequence $n_{r}:=n_{k_{r}}\rightarrow \infty$ along which we have the above almost sure convergence to the common value $\eta_{c}^{(a)}$. This shows the assertion. 
\end{proof}

\begin{remark}\label{remark:asymptotic_equivalence}
	The conclusion of Theorem \ref{thm:SLLN_rows} can be schematically summarized as 
	\begin{equation}\label{eq:asymptotic_equivalence}
	\text{BBS started with $X^{n,\mathbf{p}}$}\quad  \simeq \quad 
	\begin{matrix}
	\text{BBS started with $X^{n,\mathbf{p}}$}\\
	\text{conditioned on the highest states}
	\end{matrix}
	\end{equation}
	provided $p_{0}\ge p_{1}\ge \cdots \ge p_{\kappa}$, where $\simeq$ means that the rescaled energy matrices have the same limit. In \cite{kuniba2018randomized}, we adopted this as a hypothesis for the generalized BBS initialized at a similar product measure on $(B_{s}^{(r)})^{\mathbb{N}}$. Despite irreducibility of the corresponding carrier process in this general situation is still in question, we suspect a large deviations principle can be established in the general case as in Theorem \ref{thm:row_concentration}. Then the hypothesis of asymptotic equivalence is justified if the probability that $X^{n,\mathbf{p}}$ is a highest state decays sub-exponentially in $n$, following a similar method we used in the proof of Theorem \ref{thm:SLLN_rows} (ii). We leave verification of this for a future work.
\end{remark}

\vspace{0.3cm}

\section{Scaling limit of Young diagrams from TBA}
\label{section:TBA}

In this section, we present an alternative derivation of the scaling limit of 
the invariant Young diagrams by Thermodynamic Bethe Ansatz (TBA). 
In contrast to the Markov chain method giving the asymptotic marginal distribution of the row lengths, 
the TBA is based on the full joint distribution of the Young diagrams 
given by the {\em Fermionic form} for highest states (Theorem \ref{thm4:gibbs_measure_YDs}).
The resulting equilibrium shape of the Young diagrams is shown to agree nontrivially 
with the Markov chain method. 
Such results have been obtained in a more general setting in 
\cite{kuniba2018randomized}, 
where each site is occupied with an $r\times s$ tableau.
The content of this section corresponds to $r=s=1$ but 
includes more detailed results in a one-parameter specialization.
Namely, we will present the exact scaling form not only for the row length but
also the vacancy (defined in (\ref{a:va})), the column multiplicity, and the 
depth of the Young diagrams.
They all exhibit nontrivial factorization, which is reported here for the first time.

Throughout this section, we assume $p_{0}>p_{1}>\cdots>p_{\kappa}$ and consider the basic $\kappa$-color BBS started at a random initial configuration $X^{n,\mathbf{p}}$ conditioned on the highest state.

\subsection{Fermionic formula and grand canonical ensemble }\label{sb:tatki}

When a basic $\kappa$-color BBS is started at a highest state, extracting the $\kappa$-tuple of invariant Young diagram $(\mu^{(1)}, \cdots, \mu^{(\kappa)})$ can be done most universally by the map called Kerov-Kirillov-Reshetikhin bijection\footnote{It was invented in a very different context in 1980's and is later found \cite{kuniba2006crystal} to {\em linearize} the BBS dynamics.} \cite{kerov1988combinatorics}. The bijection is defined recursively and complicated in general, so we refer to  \cite[Section 2.7]{kuniba2006crystal} for details.

Let us introduce the following quantities (See Figure \ref{fig:TBA_notation} for illustration). 
\begin{align}
&m^{(a)}_i = \text{number of the length $i$ columns in $\mu^{(a)}$},
\label{mimi}\\
&C_{ab} = 2\delta_{a,b}-\delta_{a,b-1}-\delta_{a,b+1}, \label{Cab}\\
&E^{(a)}_i=\sum_{j \ge 1}\min(i,j)m^{(a)}_j = 
\rho_{1}^{(a)} + \rho_{2}^{(a)}+ \cdots + \rho_{i}^{(a)},
\label{Edef}\\
&v^{(a)}_i= n\delta_{a,1} - 
\sum_{b=1}^\kappa C_{ab} E^{(b)}_{i}.
\label{a:va}
\end{align}
Here $\delta_{a,b}=\mathbf{1}(a=b)$ is the Kronecker delta, the matrix $(C_{ab})_{1\le a,b\le \kappa}$ is the Cartan matrix of $sl_{\kappa+1}$, and the quantity $v^{(a)}_i$ is called \textit{vacancy}. 
Recall that (\ref{Edef}) is the defining relation of the 
multiplicity $m_{j}^{(a)}$ of the length $j$ columns of $\mu^{(a)}$, 
and $E_{i}^{(a)}=E_{i}^{(a)}(X_{0})$ means the row transfer matrix energy.

%\iffalse
\begin{figure*}[h]
	\centering
	\includegraphics[width=8cm,keepaspectratio]{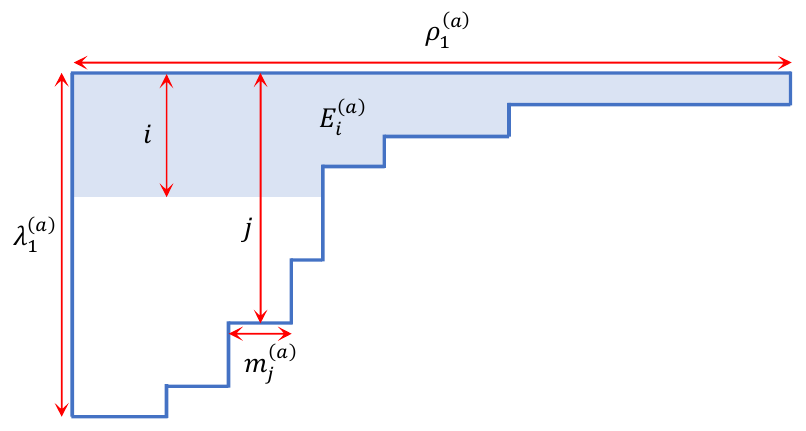}
	\caption{ Illustration of notation for the $a^{\text{th}}$ Young diagram $\mu^{(a)}$. 
	$\rho^{(a)}_i =$ length of the $i^{\text{th}}$ row,
	$m^{(a)}_j =$ multiplicity of length $j$ columns, 
	$E_{i}^{(a)} = $ area of the shaded region.
	}
	\label{fig:TBA_notation}
\end{figure*}
%\fi

It is known that the number of highest states  
corresponding to the prescribed $\kappa$-tuple of Young diagrams 
$(\mu^{(1)}, \cdots, \mu^{(\kappa)})$ can be written down explicitly as the so-called 
\textit{Fermionic form} (cf. \cite{hatayama248remarks}):
\begin{align}\label{a:ff}
\prod_{a=1}^\kappa\prod_{i \ge 1}
\binom{v^{(a)}_i + m^{(a)}_i}{m^{(a)}_i},
\end{align}
where $v_{i}^{(a)}$ and $m_{i}^{(a)}$ for the $a^{\text{th}}$ Young diagram $\mu^{(a)}$ are as in \eqref{mimi} and \eqref{a:va}, respectively. 

Since $m^{(a)}_i=0$ for sufficiently large $i$,
the apparent infinite product (\ref{a:ff}) is convergent. For (\ref{a:ff}) not to vanish, $v^{(a)}_i \ge 0$ must hold, which imposes nontrivial constraints among $\mu^{(1)}, \cdots, \mu^{(\kappa)}$. 
From \eqref{Edef} and \eqref{a:va}, we set
\begin{align}\label{einf}
v^{(a)}_\infty = n\delta_{a,1}- \sum_{b=1}^{\kappa} C_{ab}|\mu^{(b)}|,\qquad
E^{(a)}_\infty = \sum_{i\ge 1} i m^{(a)}_i = |\mu^{(a)}|.
\end{align}
From \cite{kuniba2006crystal} we recall that the size of the Young diagrams $(\mu^{(1)}, \cdots, \mu^{(\kappa)})$ are related to the color content of the BBS initial state $X_{0}$ by 
\begin{align}\label{a:ma}
|\mu^{(a)}| = \#( \text{balls of color $\ge a$ in $X_{0}$} )  \qquad \text{$\forall 1\le a \le \kappa$}.
\end{align}

Now we can derive the probability measure on 
the set of $\kappa$-tuple of Young diagrams 
$(\mu^{(1)}, \cdots, \mu^{(\kappa)})$. The following result is a special case of eq. (43) in \cite{kuniba2018randomized}, but below we include a rigorous and self-contained proof.

\begin{theorem}\label{thm4:gibbs_measure_YDs}
	For each $\kappa$-color BBS configuration $X$, let $\Sigma(X)$ denote its associated $\kappa$-tuple of invariant Young diagrams. Then for every $\kappa$-tuple of Young diagrams $(\mu^{(1)},\cdots, \mu^{(\kappa)})$, 
	\begin{equation}\label{eq:thm4_Gibbs}
	\mathbb{P}\left(\Sigma(X^{n,\mathbf{p}})=(\mu^{(1)},\cdots, \mu^{(\kappa)}) \,\bigg|\, \text{$X^{n,\mathbf{p}}$ is highest} \right) = \frac{1}{Z_{n}}
	e^{-\sum_{a=1}^\kappa \beta_a\sum_{i\ge 1}i m^{(a)}_i}
	\prod_{a=1}^\kappa \prod_{i \ge 1}
	\left({v^{(a)}_i +  m^{(a)}_i \atop m^{(a)}_i}\right),
	\end{equation}
	where the chemical potentials $\beta_{a}$ are defined by $e^{\beta_{a}}=p_{a-1}/p_{a}$ for $1\le a\le \kappa$ and $Z_{n}$ is the normalization constant.
\end{theorem}

\begin{proof}
	Let $\mathcal{H}_{n}$ denote the set of all highest $\kappa$-color BBS configuration that has support in $[0,n]$. For each $X\in \mathcal{H}_{n}$ and $0\le i \le \kappa$, we refer to the quantity $\sum_{x=0}^{n}\mathbf{1}(X(x)=i)$ as the number of $i$'s in $X$. Fix $X\in \mathcal{H}_{n}$ such that $\Sigma(X)=(\mu^{(1)},\cdots,\mu^{\kappa})$. According to \eqref{a:ma}, the number of empty boxes and the color $a$ balls contained in a BBS state are $n - |\mu^{(1)}|$ and $|\mu^{(a)}|-|\mu^{(a+1)}|$ for $a=1,\cdots, \kappa$ with $|\mu^{(\kappa+1)}|=0$. Denote $e^{\beta_{a}} = p_{a-1}/p_{a}$ for all $1\le a \le \kappa$ as in the statement of the assertion. Using \eqref{a:ma} and the second half of \eqref{einf}, we can write  
	\begin{align}
	p_{0}^{\text{$\#$ of $0$'s in $X$}}\cdots p_{\kappa}^{\text{$\#$ of $\kappa$'s in $X$}}  &= p_0^{n - |\mu^{(1)}|}p_1^{|\mu^{(1)}|-|\mu^{(2)}|} \cdots 
	p_{\kappa-1}^{|\mu^{(\kappa-1)}|-|\mu^{(\kappa)}|}
	p_\kappa^{|\mu^{(\kappa)}|} \\
		&= p_{0}^{n} e^{-\beta_1|\mu^{(1)}|-\cdots - \beta_\kappa|\mu^{(\kappa)}|}\\
		&= p_{0}^{n} e^{-\sum_{a=1}^\kappa \beta_a\sum_{i\ge 1}i m^{(a)}_i}\\
		&=  p_{0}^{n}e^{-\sum_{a=1}^\kappa E_{\infty}^{(a)} \beta_a}.
	\end{align}
	Now note that 
	\begin{align}
		&\PP\left( \Sigma(X^{n,\mathbf{p}})=(\mu^{(1)},\cdots,\mu^{(\kappa)}) \,\bigg|\, \text{$X^{n,\mathbf{p}}$ is highest} \right) \\
		&\qquad = \sum_{\substack{X\in \mathcal{H}_{n} \\ \Sigma(X)=(\mu^{(1)},\cdots, \mu^{(\kappa)})}} \PP\left( X^{n,\mathbf{p}}=X \,\bigg|\, \text{$X^{n,\mathbf{p}}$ is highest} \right) \\
		&\qquad = \sum_{\substack{X\in \mathcal{H}_{n}\\ \Sigma(X)=(\mu^{(1)},\cdots, \mu^{(\kappa)})}} \frac{\PP\left( X^{n,\mathbf{p}}(x)=X(x)\,\, \forall x\in [0,n]\right)}{\PP\left(\text{$X^{n,\mathbf{p}}$ is highest}  \right)} \\
		&\qquad = \frac{1}{\PP\left(\text{$X^{n,\mathbf{p}}$ is highest}\right)} \sum_{\substack{X\in \mathcal{H}_{n}\\ \Sigma(X)=(\mu^{(1)},\cdots, \mu^{(\kappa)})}} p_{0}^{\text{$\#$ of $0$'s in $X$}}\cdots p_{\kappa}^{\text{$\#$ of $\kappa$'s in $X$}} \\
		&\qquad = \frac{p_{0}^{n}}{\PP\left(\text{$X^{n,\mathbf{p}}$ is highest}\right)} \sum_{\substack{X\in \mathcal{H}_{n}\\ \Sigma(X)=(\mu^{(1)},\cdots, \mu^{(\kappa)})}} e^{-\sum_{a=1}^\kappa E_{\infty}^{(a)} \beta_a}.
	\end{align}
	Since the summand in the last expression does not depend on the choice of $X$ as long as $\Sigma(X)=(\mu^{(1)},\cdots,\mu^{(\kappa)})$, this and the Fermionic form \eqref{a:ff} yield 
	\begin{align}
	\PP\left( \Sigma(X^{n,\mathbf{p}})=(\mu^{(1)},\cdots,\mu^{(\kappa)}) \,\bigg|\, \text{$X^{n,\mathbf{p}}$ is highest} \right)=\frac{1}{Z_{n}} e^{-\sum_{a=1}^\kappa E_{\infty}^{(a)} \beta_a} \prod_{a=1}^\kappa \prod_{i \ge 1}
	\left({v^{(a)}_i +  m^{(a)}_i \atop m^{(a)}_i}\right)
	\end{align}
	with the partition function 
	\begin{equation}\label{gcp}
	Z_n =\sum_{m}
	e^{-\sum^\kappa_{a=1}E^{(a)}_\infty \beta_a}
	\prod_{a=1}^\kappa \prod_{i \ge 1}
	\left({v^{(a)}_i +  m^{(a)}_i \atop m^{(a)}_i}\right),
	\end{equation}
	where the outer sum extends over 
	$m =(m^{(a)}_i)_{(a,i) \in [1,\kappa]\times \in \mathbb{N}}$. This shows the assertion.
\end{proof}

Note that the right hand side of \eqref{gcp} depends on the system size $n$ via (\ref{a:va}). The parameters $\beta_1,\cdots, \beta_\kappa$ play the role of chemical potentials or inverse temperatures in the context of the generalized Gibbs ensemble.

\subsection{Y-system as the equilibrium condition}
\label{subsection:Y-system}

Our next aim is to determine the `equilibrium', i.e., most probable tuple of Young diagrams $(\mu^{(1)}, \cdots, \mu^{(\kappa)})$ under the probability distribution (\ref{eq:thm4_Gibbs}) in the large $n$ limit. It will be done by the method of grand canonical ensemble 
with the partition function given in \eqref{gcp}. 

Hereafter, for any two sequences of random variables $A_{n}$ and $B_{n}$, we write $A_{n}\sim B_{n}$ if $A_{n}/B_{n}$ converges to some constant almost surely as $n\rightarrow \infty$. In the large $n$ limit, Theorem \ref{thm:SLLN_rows} and definitions in \eqref{mimi}-\eqref{a:va} guarantee the following  asymptotic relation
\begin{align}\label{mpel}
m^{(a)}_i = \rho^{(a)}_{i}-\rho^{(a)}_{i-1} \sim n \xi^{(a)}_i,\quad
v^{(a)}_i \sim n \varphi^{(a)}_i,\quad
E^{(a)}_i  \sim n \varepsilon^{(a)}_i
\end{align}
almost surely as $n\rightarrow \infty$, where $\xi^{(a)}_i, \varphi^{(a)}_i, \varepsilon^{(a)}_i$ are suitable positive constants. Here the symbol $\eps_{i}^{(a)}$ is taken over from (\ref{eq:thm1_SLLN}). Moreover, by (\ref{a:ma}), we have 
\begin{equation}\label{eq:nu_asymptotic}
|\mu^{(a)}| \sim n (p_{a}+p_{a+1}+\cdots+p_{\kappa}). 
\end{equation}
Then after normalizing the equations (\ref{mimi})--(\ref{Edef}) by $n$, we get 
\begin{align}
&\varphi^{(a)}_i = \delta_{a,1} - \sum_{b=1}^\kappa
C_{ab}\,\varepsilon^{(b)}_i,
\quad
\varepsilon^{(a)}_i = \sum_{j \ge 1}\min(i,j)\xi^{(a)}_j,
\label{rsm1}\\
&\varphi^{(a)}_\infty = \delta_{a,1} - \sum_{b=1}^\kappa 
C_{ab}\,\varepsilon^{(b)}_\infty,
\quad
\varepsilon^{(a)}_\infty = \sum_{j=1}^{\infty} j \xi^{(a)}_j =  p_{a}+p_{a+1}+\cdots+p_{\kappa},
\label{rsm2}
\end{align}
where (\ref{rsm1}) is due to (\ref{Edef}) and $(\ref{a:va})$, and (\ref{rsm2}) follows from (\ref{einf}), (\ref{mpel}), and (\ref{eq:nu_asymptotic}).

This allows us to define the `free energy per site', which is the limit of $(-1/n)$ times the logarithm of the summand in (\ref{gcp}) as $n\rightarrow \infty$. After applying Stirling's formula, this reads 
\begin{align}\label{Fn}
F[\xi] = \sum_{a=1}^\kappa\beta_a\sum_{i=1}^{l} i\xi^{(a)}_i 
-
\sum_{a=1}^\kappa \sum_{i=1}^{l} 
\Bigl((\xi^{(a)}_i + \varphi^{(a)}_i)\log(\xi^{(a)}_i + \varphi^{(a)}_i)
- \xi_i^{(a)} \log \xi^{(a)}_i -\varphi^{(a)}_i \log \varphi^{(a)}_i\Bigr),
\end{align}
where $\xi=(\xi^{(a)}_i)$ and we have introduced the cutoff $l$ for $i$. We will take $l\rightarrow \infty$ later.  Note that the assumption (\ref{mpel}) is  consistent with the extensivity of the free energy, which enabled us 
to remove the system size $n$ as the common overall factor. 

According to Theorem \ref{thm4:gibbs_measure_YDs}, the equilibrium is achieved at $\xi$ which minimizes $F[\xi]$. Noting that $\partial \varphi^{(b)}_{j} / \partial \xi^{(a)}_{i} = -C_{ab} \min(i,j)$, the equilibrium condition 
$\partial F[\xi] /\partial \xi^{(a)}_i = 0$ is expressed as the following TBA equation
\begin{align}
-i\beta_a + \log(1+y^{(a)}_i) = 
\sum_{b=1}^\kappa C_{ab}\sum_{j=1}^{l}\min(i,j)
\log(1+(y^{(b)}_j)^{-1})
\label{TBAn}
\end{align}
in terms of the ratio 
\begin{equation}\label{yf}
y^{(a)}_i = \varphi^{(a)}_{i}/\xi^{(a)}_{i}  \quad \forall 1 \le a \le \kappa.
\end{equation}
By taking double difference, this is shown to be equivalent to the Y-system
\begin{align}\label{YY}
\frac{(1+y^{(a)}_{i})^2}{
	(1+y^{(a)}_{i-1})(1+y^{(a)}_{i+1})}
= \prod_{b=1}^\kappa
(1+(y^{(b)}_i)^{-1})^{C_{ab}}
\end{align}
with the boundary condition
\begin{align}\label{ybd}
y^{(a)}_0 =0, \quad 1+y^{(a)}_{l+1} = e^{\beta_a}(1+y^{(a)}_l).
\end{align}
The Y-system is known to follow from the 
Q-system 
\begin{align}\label{qsys}
(Q^{(a)}_i)^2 
=Q^{(a)}_{i-1}Q^{(a)}_{i+1} + \prod_{C_{ab}=-1} Q^{(b)}_i,
\end{align}
by the substitution (see \cite[Prop. 14.1]{kuniba2011t}) 
\begin{align}\label{yqw}
y^{(a)}_i = \frac{Q^{(a)}_{i-1}Q^{(a)}_{i+1}}
{\prod_{C_{ab}=-1} Q^{(b)}_i},
\qquad 
1+y^{(a)}_i = \prod_{b=1}^\kappa (Q^{(b)}_i)^{C_{ab}},
\qquad
1+(y^{(a)}_i)^{-1}
= \frac{(Q^{(a)}_i)^2}{Q^{(a)}_{i-1}Q^{(a)}_{i+1}}.
\end{align}
The solution of the Q-system satisfying $Q^{(1)}_0=\cdots = Q^{(\kappa)}_0=1$ contains $\kappa$ parameters $z_{1},\cdots,z_{\kappa}$ and is given by the Schur function
\begin{align}\label{eq:def_Q}
Q_{i}^{(a)}(z_{1},\cdots,z_{\kappa}) = s_{(i^{a})} (w_{1},\cdots,w_{\kappa+1}), 
\end{align}
where we set 
\begin{align}\label{wza}
w_{a} = z_a/z_{a-1} \,\, \forall 1\le a \le \kappa+1,\qquad z_{0}=z_{\kappa+1}=1.
\end{align}
This is the character of the irreducible $sl_{\kappa+1}$ module 
associated with the 
Young diagram of $a \times i$ rectangular shape.
The property $Q^{(a)}_i \in \Z[z^{\pm 1}_1,\ldots, z^{\pm 1}_\kappa]$ holds.
The simplest one reads
\begin{align}\label{q11}
Q^{(1)}_1 = \sum_{a=0}^{\kappa}
\frac{z_{a+1}}{z_a}.
\end{align}
To validate the Q-system (\ref{qsys}) at $i=0$, we set 
$Q^{(a)}_{-1}=0$.

Next we take the boundary condition (\ref{ybd}) into account.
The left one $y^{(a)}_0=0$  
is automatically satisfied due to $Q^{(a)}_{-1}=0$.
On the other hand the right condition in (\ref{ybd}) is expressed as
\begin{align}\label{qoq}
e^{\beta_a} = \prod_{b=1}^\kappa\left(
\frac{Q^{(b)}_{l+1}}{Q^{(b)}_l}\right)^{C_{ab}}.
\end{align}

Suppose the parameters $z_1,\cdots, z_\kappa>0$ are chosen so that 
\begin{equation}\label{eq:parameter_assumption}
\prod_{b=1}^\kappa z_b^{C_{ab}}>1 \quad \forall 1\le a \le \kappa.
\end{equation}
Then one can invoke the result 
\cite[Th. 7.1 (C)]{hatayama248remarks}. In the present notation it says  
\begin{equation}\label{eq:Q_limit_ratio}
\lim_{l \rightarrow \infty}
(Q^{(a)}_{l+1}/Q^{(a)}_l)= z_a \qquad \forall 1\le a \le \kappa.
\end{equation}
Thus the large $l$ limit of (\ref{qoq}) can be taken, giving 
\begin{align}\label{mz}
e^{\beta_a} = \prod_{b=1}^\kappa z^{C_{ab}}_b,\qquad
z_a = \exp\left( \sum_{b=1}^\kappa (C^{-1})_{ab}\beta_b\right)
= \exp\left( \sum_{b=1}^\kappa \bigl(\min(a,b)-\frac{ab}{\kappa+1}\bigr)
\beta_b\right).
\end{align}
Noting the definitions of $\beta_{a}$ given in Theorem \ref{thm4:gibbs_measure_YDs}, one finds that the variables $z_1,\ldots, z_\kappa$ are related to the ball density as
\begin{align}\label{a:zp}
z_a = u^{-\frac{a}{\kappa+1}}p_0p_1\cdots p_{a-1},
\quad 
u= p_0p_1\cdots p_\kappa\qquad (1 \le a \le \kappa).
\end{align} 
From this it follows that the assumption we made in (\ref{eq:parameter_assumption}) is equivalent to $p_{0}>p_{1}>\cdots>p_{\kappa}$, which is our grounding assumption in this section.

In what follows, the quantities like 
$\xi^{(a)}_i, \varphi^{(a)}_i, \varepsilon^{(a)}_i$ are to be understood as 
the equilibrium values.

\subsection{Equation of state and consistency}

Here we derive the equation of state of the system and show that it is 
indeed satisfied by (\ref{a:zp}) to demonstrate the consistency of 
the analysis.
First we calculate the equilibrium value of the free energy per site (\ref{Fn}).
Using (\ref{yf}) rewrite (\ref{Fn}) as
\begin{align}\label{F1}
F[\xi_{\mathrm{eq}}] = \sum_{a=1}^\kappa
\beta_a\sum_{i=1}^l i\xi^{(a)}_i 
-\sum_{a=1}^\kappa \sum_{i=1}^l 
\Bigl(\xi^{(a)}_i\log(1+y^{(a)}_i) 
+ \varphi^{(a)}_i\log(1+(y^{(a)}_i)^{-1})\Bigr).
\end{align}
On the other hand taking the linear combination of the TBA equation as
$\sum_{a=1}^\kappa 
\sum_{i=1}^l (\ref{TBAn}) \times \xi^{(a)}_i$ we get
\begin{align}
\sum_{a=1}^\kappa\beta_a\sum_{i=1}^li\xi^{(a)}_i =
\sum_{a=1}^\kappa\sum_{i=1}^l\xi^{(a)}_i\log(1+y^{(a)}_i)
-\sum_{a,b=1}^\kappa C_{ab}\sum_{i,j=1}^l\min(i,j)
\xi^{(a)}_i\log(1+(y^{(b)}_j)^{-1}).
\end{align}
Substituting this into the first term on the RHS of (\ref{F1}) and 
using $\varphi^{(a)}_i$ from (\ref{rsm1}) we find
\begin{align}\label{Fm}
F[\xi_{\mathrm{eq}}] = -\sum_{i=1}^l
\log(1+(y^{(1)}_i)^{-1}) = 
-\log\left(\frac{Q^{(1)}_1Q^{(1)}_l}{Q^{(1)}_{l+1}}\right)
\overset{l \rightarrow \infty}{\longrightarrow}
-\log\left(z^{-1}_1Q^{(1)}_1\right),
\end{align}
where (\ref{yqw}) is used and the last step is due to (\ref{eq:Q_limit_ratio}).

Now we resort to the general relation
\begin{align}\label{FZ}
F[\xi_{\mathrm{eq}}]
= -\lim_{n \rightarrow \infty}\frac{1}{n}\log Z_n.
\end{align}
From (\ref{einf}), (\ref{gcp}) and (\ref{eq:nu_asymptotic}) 
one has the equation of state 
$p_a+p_{a+1}+\cdots+p_{\kappa} = \frac{\partial F[\xi_{\mathrm{eq}}]}{\partial \beta_a}$ for 
$1 \le a \le \kappa$.
In view of (\ref{mz}) 
it is convenient to take the linear combination of it as follows:
\begin{align}\label{est}
\sum_{b=1}^\kappa C_{ab}(p_{b}+p_{b+1}+\cdots+p_{\kappa}) = 
\left(\sum_{b=1}^\kappa C_{ab}\frac{\partial}{\partial \beta_{b}}\right)
F[\xi_{\mathrm{eq}}]
= z_a\frac{\partial F[\xi_{\mathrm{eq}}]}{\partial z_a}.
\end{align}
Substituting (\ref{Fm}) we obtain the 
equation of state: 
\begin{equation}\label{est}
z_a \frac{\partial}{\partial z_a}\log Q^{(1)}_1= 
\delta_{a,1} - \sum_{b=1}^\kappa C_{ab} (p_{b}+p_{b+1}+\cdots+p_{\kappa}).
\end{equation}
From (\ref{q11}), it reads explicitly as
$(z_0 = z_{\kappa+1}=1)$
\begin{equation}\label{est2}
z_a \frac{\partial}{\partial z_a}\log \Bigl(
\sum_{b=0}^{\kappa}
\frac{z_{b+1}}{z_b}\Bigr)= 
\delta_{a,1} - \sum_{b=1}^\kappa C_{ab}(p_b + p_{b+1} + \cdots + p_\kappa).
\end{equation}
These $\kappa$ equations relate the two sets of variables $p_1, \cdots, p_\kappa$
and $z_1, \cdots, z_\kappa$.
The former are densities of balls in the BBS while 
the latter are essentially the fugacities 
$e^{-\beta_1}, \cdots, e^{-\beta_n}$ as seen in (\ref{mz}). 
It is an elementary exercise to check that the substitution (\ref{a:zp})
satisfies (\ref{est2}) under the normalization condition 
$p_0 + p_1+ \cdots + p_\kappa=1$.
The relation (\ref{est2}) is also consistent with $e^{\beta_a}=p_{a-1}/p_a$
given in Theorem \ref{thm4:gibbs_measure_YDs}.

\subsection{Difference equation characterizing equilibrium Young diagrams}

Recall from (\ref{Edef}) and (\ref{mpel}) that
\begin{align}\label{epm}
\varepsilon^{(a)}_i &= \lim_{n \rightarrow \infty}\frac{1}{n}
(\text{$\#$ of boxes in the first $i$ rows of $\mu^{(a)}$}).
\end{align}
Therefore in order to determine the scaled Young diagrams in the
equilibrium, it suffices to characterize $\varepsilon^{(a)}_i$ 
or $\varphi^{(a)}_i$ related to it by (\ref{rsm1}) 
for $(a,i) \in [1,\kappa]\times \Z_{\ge 1}$.
This is done as follows.
Taking the double difference of the first relation in (\ref{rsm1})
leads, by using (\ref{yf}), to 
\begin{align}\label{deq}
&\varphi^{(a)}_{i-1}-2\varphi^{(a)}_i + \varphi^{(a)}_{i+1}
= \sum_{b=1}^\kappa C_{ab} (y^{(b)}_i)^{-1}\varphi^{(b)}_i
\qquad (i \ge 1, \, a \in [1,\kappa]).
\end{align}
Here we have introduced 
$\varphi^{(a)}_0=\delta_{a,1}$. 
Since (\ref{deq}) is second order as a difference equation 
with respect to $i$, we are yet to specify $\varphi^{(a)}_i$
at another $i$. 
Such a point is available at $i=\infty$ from (\ref{rsm2}). 
This completes a characterization of the scaled vacancy 
$\varphi^{(a)}_i$ hence the equilibrium Young diagrams 
$\mu^{(1)}, \cdots, \mu^{(\kappa)}$ in the scaling limit.

The procedure is summarized as follows:

\begin{enumerate}
	
	\item Given the ball densities in the region
	$1>p_0 > \cdots > p_\kappa>0$, specify 
	$z_1, \ldots, z_n$ as (\ref{a:zp}) and 
	calculate $y^{(a)}_i$ by (\ref{yqw})- (\ref{wza})  and (\ref{qf}).
	
	\item Find the unique solution to (\ref{deq}) satisfying the 
	boundary condition:
	\begin{align}\label{bcd}
	\varphi^{(a)}_0=\delta_{a,1},\qquad
	\varphi^{(a)}_\infty  = 
	\delta_{a,1}-\sum_{b=1}^\kappa C_{ab}
	(p_b + p_{b+1}+ \cdots + p_\kappa).
	\end{align}
	
	\item The quantity (\ref{epm}) is determined as  
	$\varepsilon^{(a)}_i = \sum_{b=1}^\kappa
	(C^{-1})_{ab}(\delta_{b,1}-\varphi^{(b)}_i)$.
	
\end{enumerate}

Admittedly the second step (ii) in the above is nontrivial since the 
boundary condition is imposed at the most distant two points 
$i=0$ and $i=\infty$. However, notice that according to the asymptotic equivalence (see Remark \ref{remark:asymptotic_equivalence}), the solution should be given by the stationary local energy \eqref{eq:formula_stationary_energy} obtained by the Markov chain method. In \cite{kuniba2018randomized}, this observation has been generalized for the most general BBS using the crystal base theory.

In the next subsection we present the explicit answer to the step (ii)  
for a special choice of $p_0, \ldots, p_\kappa$,
which deserves an independent demonstration in the light of 
the neat factorization not happening in the general case.
It yields an alternative proof of Corollary 
\ref{cor:principal_specialization_kappa1} 
based on the TBA analysis.

\subsection{Explicit solution under principal specialization}

Fix a parameter $0 < q <1$ and consider the 
one parameter specialization of the 
$p_0, \cdots, p_\kappa$ as 
\begin{align}
&p_a = \frac{q^{a}}{1+q+\cdots + q^\kappa} 
= \frac{q^{a}(1-q)}{1-q^{\kappa+1}}.
\end{align}
This certainly gives a probability measure on $\{0,1,\cdots,\kappa+1 \}$ and  fulfills the condition $1>p_0 > \cdots > p_\kappa>0$. Then the equation of state (\ref{est2}) is satisfied by 
\begin{align}\label{qdq}
z_a=q^{-a(\kappa+1-a)/2},
\end{align}
and such choice of $z_a$ is known as the {\em principal specialization} reducing characters to the so-called {\em $q$-dimensions}.

Under this specialization, $Q^{(a)}_i$ in \eqref{eq:def_Q} becomes 
$Q^{(a)}_i = \prod_{b=1}^a\prod_{j=1}^i
\frac{[\kappa+1+j-b]}{[a-b+i-j+1]}$, with $[x]=q^{x/2}-q^{-x/2}$.
This already brings (\ref{eq:thm1_eq1}) into the factorized form (\ref{eq:p_specialization}).
What we will do in the rest of the section is not only to 
reconfirm (\ref{eq:p_specialization}) in the current TBA framework 
but also to provide additional results which are 
explicit formulas for the scaled vacancy $\varphi^{(a)}_i$, 
the column multiplicity $\xi^{(a)}_i$ and the scaling behavior of the 
length $\lambda^{(a)}_1$ of the first column.

The quantity  $y^{(a)}_i$ (\ref{yqw})  in the principal specialization reads 
\begin{align}\label{yvv}
y^{(a)}_i = \frac{q^{-i}(1-q^i)(1-q^{i+\kappa+1})}{(1-q^a)(1-q^{\kappa+1-a})}.
\end{align}
It remains to solve (\ref{deq}) containing the above $y^{(a)}_i$'s  in the coefficients
with the boundary condition (\ref{bcd}).
The solution, established by a direct calculation, is given by
\begin{proposition}\label{pr:res}
	The following satisfies the difference equation (\ref{deq})
	and the boundary condition (\ref{bcd}):
	\begin{align}
	\varphi^{(a)}_i &=
	\frac{q^{a-1}(1-q)^2(1-q^i)(1-q^{i+\kappa+1})(1+q^{i+a})}
	{(1-q^{\kappa+1})(1-q^{i+a-1})(1-q^{i+a})(1-q^{i+a+1})}.
	\label{sigr}
	\end{align}
\end{proposition}

Unlike the principally specialized characters,
a conceptual explanation of this neat factorization is not known to the authors. 
As a corollary of Proposition \ref{pr:res}, one can also deduce the scaling form
$\xi^{(a)}_i$ of the column multiplicity $m^{(a)}_i$ as
\begin{align}
\xi^{(a)}_i =
\frac{q^{i+a-1}(1-q)^2(1-q^a)(1-q^{\kappa+1-a})(1+q^{i+a})}
{(1-q^{\kappa+1})(1-q^{i+a-1})(1-q^{i+a})(1-q^{i+a+1})}.
\end{align}
The factorized nature is inherited from (\ref{sigr}) via 
(\ref{yf}) and (\ref{yvv}). 

Now Corollary \ref{cor:principal_specialization_kappa1} immediately follows. 

\begin{proof}[\textbf{Proof of Corollary \ref{cor:principal_specialization_kappa1}}.]
	According to Theorem \ref{thm:SLLN_rows}, it suffices to show Corollary \ref{cor:principal_specialization_kappa1} (ii). By Theorem \ref{thm:SLLN_rows} and the relation (\ref{Edef}), we get 
	\begin{equation}
	\lim_{n\rightarrow \infty} n^{-1} \rho_{i}^{(a)}(X^{n,\mathbf{p}}) = \eps_{i}^{(a)} - \eps_{i-1}^{(a)}. 
	\end{equation}
	By using the first relation in (\ref{rsm1}), we then obtain 
	\begin{align}
	\lim_{n\rightarrow \infty} n^{-1} \rho_{i}^{(a)}(X^{n,\mathbf{p}}) &= 
	\sum_{b=1}^\kappa(C^{-1})_{ab}(\varphi^{(b)}_{i-1}-\varphi^{(b)}_i) \\
	& =\frac{q^{i+a-1}(1-q)(1-q^a)(1-q^{\kappa+1-a})} 
	{(1-q^{\kappa+1})(1-q^{i+a-1})(1-q^{i+a})},
	\label{rst}
	\end{align}
	where the last step is due to (\ref{sigr}) and 
	$(C^{-1})_{ab} = \min(a,b)-\frac{ab}{\kappa+1}$ which has already 
	appeared in (\ref{mz}).
\end{proof}	

\begin{remark}
	Observe that in the limit $q \rightarrow 1$ from below, we approach 
	the uniform distribution $p_0=\cdots = p_\kappa=\frac{1}{\kappa+1}$,
	and (\ref{rst}) tends to the rational form
	\begin{align}\label{rst2}
	\lim_{n\rightarrow \infty} n^{-1} \rho_{i}^{(a)}(X^{n,\mathbf{p}}) &= \frac{a(\kappa+1-a)}
	{(\kappa+1)(i+a-1)(i+a)}.
	\end{align}
	When $i=1$ the result (\ref{rst}) agrees with the value 
	obtained by a direct calculation of the 
	Markov process of the carries induced by interaction with 
	the random BBS states. 
	When $\kappa=a=1$, the formulas 
	(\ref{rst}) and (\ref{rst2}) reproduce the quantity 
	$\eta^{(1)}_i$ in (\ref{eq:thm1_eq3}).

	To make another observation, consider the sum
	\begin{align}
	\sum_{i \ge \lambda}\xi^{(a)}_i = 
	\frac{q^{\lambda+a-1}(1-q)(1-q^a)(1-q^{\kappa+1-a})}
	{(1-q^{\kappa+1})(1-q^{\lambda+a-1})(1-q^{\lambda+a})},
	\end{align}
	where $\lambda$ is a $\mathbb{N}$-valued parameter.
	The scaling behavior of the length 
	$\lambda_1^{(1)}, \ldots, \lambda_{1}^{(\kappa)}$ of the first columns of the equilibrium Young diagrams $\mu^{(1)}, \cdots, \mu^{(\kappa)}$  
	may be inferred by setting 
	\begin{equation}
	n \sum_{i \ge \lambda_1^{(a)}}\xi^{(a)}_i \sim 1.
	\end{equation}
	This postulate leads to a crude estimate 
	\begin{align}
	\lambda_1^{(a)} \sim 
	-\frac{1}{\log q}\log\left(
	\frac{q^{a-1}(1-q)(1-q^a)(1-q^{\kappa+1-a})n}{1-q^{\kappa+1}}\right)
	\sim -\frac{\log n}{\log q}\quad (n \rightarrow \infty).
	\end{align}
	When $\kappa=a=1$, 
	this logarithmic scaling reproduces $\mu_{n}$ in 
	\cite[Th.2 (i)]{levine2017phase} as
	$\mu_{n} = \lambda_1^{(1)}|_{q=\theta^{-1}}$.
\end{remark}

\vspace{0.3cm}

\vspace{0.3cm}
\section*{Acknowledgments}
The authors appreciate valuable conversations with Frank Aurzada, Mikhail Lifshits, Masato Okado, and Makiko Sasada. Atsuo Kuniba is supported by 
Grants-in-Aid for Scientific Research No.~18H01141 from JSPS.

\vspace{0.3cm}

\small{
\newcommand{\etalchar}[1]{$^{#1}$}
\providecommand{\bysame}{\leavevmode\hbox to3em{\hrulefill}\thinspace}
\providecommand{\MR}{\relax\ifhmode\unskip\space\fi MR }
\providecommand{\MRhref}[2]{%
	\href{http://www.ams.org/mathscinet-getitem?mr=#1}{#2}
}
\providecommand{\href}[2]{#2}

}

\end{document}